\documentclass{article}

\usepackage[utf8]{inputenc}
\usepackage{amsmath}
\usepackage{amsfonts}
\usepackage{color}
\usepackage{latexsym}
 \usepackage[colorlinks]{hyperref}
\usepackage{tablefootnote}
\usepackage{algpseudocode}
\usepackage{enumerate}
\usepackage{algorithm}

\newcommand*{\QEDA}{\hfill\hbox{\vrule width1.0ex height1.0ex}}
\newcommand*{\E}{\mathbb{E}}
\usepackage{multirow}
\usepackage{amsmath}
\usepackage{amsfonts}
\usepackage{latexsym}
\usepackage{amssymb}

\newtheorem{thm}{Theorem}[section]
\newtheorem{theorem}[thm]{Theorem}

\newtheorem{lemma}[thm]{Lemma}

\newtheorem{proposition}[thm]{Proposition}

\newtheorem{corollary}[thm]{Corollary}
\newtheorem{definition}[thm]{Definition}

\newtheorem{remark}[thm]{Remark}

\newcommand{\beq}{\begin{equation}}
\newcommand{\eeq}{\end{equation}}
\newcommand{\beqa}{\begin{eqnarray}}
\newcommand{\eeqa}{\end{eqnarray}}
\newcommand{\beqas}{\begin{eqnarray*}}
\newcommand{\eeqas}{\end{eqnarray*}}
\newcommand{\bi}{\begin{itemize}}
\newcommand{\ei}{\end{itemize}}

\newcommand{\vgap}{\vspace{.1in}}
\newcommand{\nn}{\nonumber}

\newcommand{\R}{\mathbb{R}}

\newcommand{\lam}{{\lambda}}

\newcommand{\norm}[1]{\left\Vert#1\right\Vert}

\newcommand{\inner}[2]{\langle #1,#2\rangle}

\newcommand{\argmin}{\mathrm{argmin}\,}

\newcommand{\dom}{\mathrm{dom}\,}

\newcommand{\Argmin}{\mathrm{Argmin}\,}
\newcommand{\dist}{\mathrm{dist}\,}

\newcommand{\Conv}[1]{\mbox{\rm Conv}(\R^{#1})}

\newcommand{\bConv}[1]{\overline{\mbox{\rm Conv}}\,(\R^{#1})}

\newcommand{\tz}{\tilde z}

\begin{document}
\title{Proximal bundle methods
	for hybrid weakly convex composite optimization problems}
\author{
		Jiaming Liang \thanks{
        Goergen Institute for Data Science and Artificial Intelligence and Department of Computer Science, University of Rochester, Rochester, NY 14620 (email: {\tt jiaming.liang@rochester.edu}). This work was partially supported by AFOSR grant FA9550-25-1-0182. }\qquad 
        Renato D.C. Monteiro \thanks{School of Industrial and Systems
		Engineering, Georgia Institute of
		Technology, Atlanta, GA, 30332-0205.
		(email: {\tt renato.monteiro@isye.gatech.edu} and {\tt hzhang906@gatech.edu}). This work
			was partially supported by AFOSR Grants FA9550-22-1-0088 and FA9550-25-1-0131.}\qquad 
	Honghao Zhang \footnotemark[2]
	}
	\date{March 26, 2023 (first revision: December 26, 2024)}
% \newcounter{savecntr}% Save footnote counter
% \newcounter{restorecntr}% Restore footnote counter
% \author{Gopal \setcounter{savecntr}{\value{footnote}}\thanks{XYZ University} ,
% Someone else \thanks{ABC University} ,
% Another \setcounter{restorecntr}{\value{footnote}}%
%   \setcounter{footnote}{\value{savecntr}}\footnotemark% Print footnotemark
%   \setcounter{footnote}{\value{restorecntr}}}

	\date{March 26, 2023 \\
    (1st revision: Dec 26, 2024; 2nd revision: May 16, 2026; 3rd revision: Oct 13, 2026)}
\maketitle

\begin{abstract}
       This paper establishes the iteration-complexity of proximal bundle methods for solving
 hybrid (i.e., a blend of smooth and nonsmooth) weakly convex composite optimization (HWC-CO) problems.
 This is done in a unified manner by considering a proximal bundle framework (PBF), which includes various well-known bundle update schemes.
In contrast to
hard-to-check stationary conditions 
 (e.g., the Moreau  stationarity)
 used by other methods
for solving HWC-CO, PBF uses a stationarity measure
% in the context of HWC-CO
that is easily verifiable.
% Finally, computational results comparing PBF to a well-known subgradient method  show that it can solve many HWC-CO instances in substantially less CPU time than the latter one.
   % For example, in less than 20 minutes, \ourmethod can solve a maximum stable set SDP instance with dimension pair $(n,m)\approx (10^6,10^7)$ within 
   % $10^{-5}$ relative precision.
% \red{Instead of using
%  other well-known stationary conditions 
%  (e.g., the Moreau  stationary one)
% % in the context of
% % HWC-CO,
% PBF uses a new stationary measure which is a direct extension of $\varepsilon$ subdifferentail for convex function 
%  in the context of HWC-CO. The new stationary measure is directly verifiable and,
% at the same time, implies any of the former ones.}

       % a proximal bundle framework (PBF) for hybrid weakly convex composite optimization (HWC-CO) problem. We obtain the first global complexity result of proximal bundle method for HWC-CO. A major advantage of our analysis is that it bases on a newly defined convergence measure which is the first checkable notion of approximate  stationarity in the literature.
       
		\\
		
		% 	{\bf Mathematics Subject Classification (2010)} 
		% 	49M37 $\cdot$ 65K05 $\cdot$ 68Q25 $\cdot$ 90C25 $\cdot$ 90C30 $\cdot$ 90C60
	\end{abstract}
	 
\section{Introduction}
	Let  $h: \R^{n} \rightarrow \R\cup \{ +\infty \} $ be a proper lower semi-continuous convex function
 and $ f: \R^{n} \rightarrow \R\cup \{ +\infty \} $ be a proper lower semi-continuous $m$-weakly convex function (i.e.,
 $f+(m/2)\|\cdot\|^2$
 is convex)
 such that
	$ \dom f \supseteq \dom h$  and
	consider
	the composite optimization (CO) problem
	\begin{equation}\label{eq:opt_problem}
	\min \left\{\phi(x):=f(x)+h(x): x \in \R^n\right\}.
	\end{equation}
	It is said that \eqref{eq:opt_problem} is a hybrid weakly convex CO (HWC-CO) problem if there exist
	nonnegative scalars $M$ and $L$
	and a
	first-order oracle $f':\dom h \to \R^n$
	(i.e., $f'(x)\in \partial f(x)$
	for every $x \in \dom h$)
	satisfying the $(M,L)$-hybrid condition,
	namely:
	$\|f'(u)-f'(v)\| \le 2M + L \|u-v\|$ for every $u,v \in \dom h$.
 This problem class includes the class of weakly convex
 non-smooth (resp., smooth) CO problems, i.e., the one with $L=0$
 (resp., $M=0$).
Throughout our complexity analysis, we assume $m>0$. Thus, the convex case
$m=0$ is not formally covered by the complexity results.

 This problem class appears in various applications in modern data science where $f$ is usually a loss function and $h$ is either the indicator function of some set (e.g., the set of points satisfying some functional constraints) or a regularization function that imposes sparsity or some special structure on the solution being sought.
Examples of such applications are robust phase retrieval, covariance matrix estimation, and sparse dictionary learning
(see Subsection 2.1 of \cite{davis2019stochastic} and the references therein).   In these applications, the loss is often weakly
convex while the regularizer is convex but nonsmooth, making the HWC-CO
framework a natural modeling setting.

 The main goal of this paper is to study the complexity of
  a unified framework (referred to as PBF) of  
 proximal bundle (PB) methods for solving the HWC-CO problem \eqref{eq:opt_problem}.
More specifically,
like 
 other proximal bundle methods, a PBF iteration solves a prox bundle subproblem
	of the form
 \begin{equation}\label{eq:x-pre}
	    x = \underset{u\in \R^n}\argmin \left \{\Gamma (u) + \frac{1}{2\lam} \|u-x^c\|^2 \right\}
	\end{equation}
	where $\lam$ is the prox stepsize,
	$x^c$ is
 the current
	prox-center,
 and (usually simple) lower semi-continuous convex function $\Gamma$ denotes  the current bundle function.
	Moreover,
 the bundle function is updated in every iteration but the prox-center $x^c$ is updated in some
 (i.e., serious)
 iterations and left the same
 in the other (i.e., null) ones.
Instead of choosing the bundle function underneath $\phi$ as in the proximal bundle  methods for solving the convex version of \eqref{eq:opt_problem},
 PBF uses the idea of choosing $\Gamma$ underneath $\phi(\cdot)+(m/2)\|\cdot-x^c\|^2$.
 An interesting
 feature of
 PBF is that it terminates based on
 an easily verifiable stopping criterion. 

{\it Literature Review:}
We first discuss the development of PB methods for convex
CO problems. PB methods were introduced in
\cite{lemarechal1975extension,wolfe1975method} and subsequently developed
in, for example,
\cite{frangioni2002generalized,mifflin1982modification,de2014convex}.
Their convergence analyses for nonsmooth convex problems are discussed
extensively in the textbooks
\cite{ruszczynski2011nonlinear,urruty1996convex}, while their iteration
complexity has been studied in, for example,
\cite{diaz2023optimal,du2017rate,kiwiel2000efficiency,
liang2021proximal,liang2024unified}.
For broader treatments of bundle methods, including methods for nonconvex
nonsmooth optimization, we also refer to the recent textbooks
\cite{bagirov2020numerical,van2025methods}.

We next discuss PB methods for weakly convex CO problems, as introduced
at the beginning of this section, focusing first on the nonsmooth case
$L=0$. Extensions of PB methods to this setting can be found in
\cite{fuduli2004minimizing,hare2009computing,hare2010redistributed,
hare2016proximal,kiwiel2006methods,makela1992nonsmooth,
noll2008proximity,vlvcek2001globally}. These works primarily establish
asymptotic convergence results. In particular,
\cite{hare2009computing,hare2010redistributed,hare2016proximal}
already considered constructing convex bundle (more precisely,
cutting-plane) models underneath the regularized function
$\phi(\cdot)+(m/2)\|\cdot-x^c\|^2$,
which is also the basic convexification principle used by PBF. Thus, the
distinction of PBF does not lie in this convexification principle itself.
Rather, PBF develops a general proximal bundle framework based on this
principle and establishes explicit overall iteration-complexity bounds for
obtaining regularized stationary points under the hybrid condition
considered in this paper. The more recent work
\cite{atenas2023unified} develops a unified model for analyzing
descent-type bundle methods and establishes local convergence rates for
the serious iterate and function value sequences under a strong stationary
growth condition. In contrast, the focus of the present work is on global
nonasymptotic iteration-complexity guarantees.

Weakly convex problems also admit a natural difference-of-convex (DC)
representation, and hence PBF is related to bundle methods developed for
nonsmooth DC programming; see, for example,
\cite{de2019proximal,joki2017proximal,kanzow2024bundle}.
We discuss this connection, and the distinction between PBF and these
methods, in more detail in Subsection~\ref{SEC:SCS}.

{\it Main contribution.}
As already mentioned above, this paper establishes the (serious and overall) iteration-complexity of PB methods for solving the HWC-CO problem described at the beginning of this introduction.  An interesting feature of our analysis is that it is based on a stationarity measure
% in the context of HWC-CO
which, in contrast to other well-known stationary conditions
in the context of
HWC-CO (including the Moreau stationary one), is easily and directly verifiable. Moreover, by proper choice of tolerances, it is shown that the new measure
implies these
other
well-known near stationarity conditions (including the Moreau stationary one), so that any complexity result
based on the first one can be easily translated to the latter ones.

As a consequence of our analysis,
it is shown that the iteration-complexity for PBF to find a $\rho$-Moreau stationary point
of $\phi$ 
% {\bf Definition????}
is similar to that of the deterministic versions of
the stochastic proximal subgradient (PS) methods studied 
in \cite{davis2019stochastic,davis2019proximally},
i.e., ${\cal O}(\rho^{-4})$,
 but the first complexity bound has the  advantage that its constant (in its ${\cal O}(\cdot)$) is never worse and is generally better than the one which appears in the bound for the PS method
of \cite{davis2019stochastic}.
The latter feature of PBF is due to its
bundle nature which allows it to
use a considerably  larger prox stepsize that is determined by
the weakly convex parameter.
This contrasts with the nature of proximal subgradient-type methods which
use relatively small prox stepsizes
(e.g., depending on a pre-specified iteration count or decreasing towards zero as the latter grows).

Other advantages of our method are:
1) PBF assumes that $f$ satisfies the more general hybrid condition (both
\cite{davis2019stochastic,davis2019proximally} only consider the case where $L=0$);
2) as opposed to the methods of \cite{davis2019stochastic,davis2019proximally},
 PBF does not require knowledge of
the constants $M$ and $L$ (see the first paragraph above), which are usually hard to estimate, however, PBF still requires knowledge of a weak convexity parameter $m$;
% but as the first two methods, PBF still requires knowledge of a weakly convex parameter $m$;
% \blue{PBF does require knowledge of a weak convexity parameter $m$; however,
% \cite{davis2019proximally} also presents a parameter-free variant that does not require prior knowledge of $m$, at the expense of a slower convergence rate;}
and 3)
it can obtain an estimate of the size of the gradient of the Moreau envelope whenever  the prox subproblem is solved according to an easily verifiable stopping criterion;
in contrast, 
Algorithm 2 of \cite{davis2019proximally} needs to
perform a pre-specified number of
subgradient iterations to the prox subproblem
(see Section 3 of \cite{davis2019proximally}) to estimate the
size of the gradient of the Moreau envelope.

% by simply obtaining a without performing a pre-specified number of inner number to the obtain an approximation solution of the prox subproblem.

% of 
% More specifically, it is shown
% that every iteration of a
% PBF variant generates a pair $(x,w)$ of vectors and a scalar $\varepsilon$ such that
% $w \in \partial_{\varepsilon}(\cdot;x^c)(x)$ and the stationary 
% It is shown in this paper that
% any instance of PBF finds a

% \begin{itemize}
%     \item We propose a new notion of stationary that is checkable and can used for other algorithms as a convergence measure for weakly convex nonsmooth optimization problem.
%     \item Based on the new convergence measure, we obtain the first global complexity result of proximal bundle method for HWC-CO which includes weakly convex nonsmooth problem as a special case under a very weak condition. We also show that our complexity for PBF is better than the  complexity for stochastic composite subgradient method proposed in \cite{davis2019stochastic}. 
% \end{itemize}

{\it Organization of the paper.}
Section~\ref{Sec:background} contains three subsections.
Subsection~\ref{subsec:DefNot}  presents basic definitions and notation used throughout the paper.
Subsection~\ref{subsec:notion} presents three stationary notions and clarifies the relationships among these three notions.
Subsection~\ref{SEC:SCS}
reviews related methods for weakly convex nonsmooth optimization.
Section~\ref{Sec:PBF} contains three subsections. Subsection~\ref{sec:assumptions} formally
describes problem~\eqref{eq:opt_problem} and the assumptions made on it. Subsection~\ref{subsec:update} describes PBF and states the iteration-complexity result of PBF. Subsection~\ref{sec:BUF} presents two concrete instances of PBF.  Section~\ref{sec:proof} contains three subsections. Subsection \ref{subsec:length} provides a preliminary bound on the length of a cycle in PBF and Subsection \ref{subsec:cycle} bounds the number of cycles generated by PBF. Subsection \ref{subsec:proof} presents the proof of the main complexity result.
  Section~\ref{num} presents computational results for three problems: phase retrieval, blind deconvolution, and constrained MAXQUAD.
Section~\ref{sec:conclusion} gives some concluding remarks and potential directions for future research.
Appendix~\ref{APP:sub} describes two useful technical results about subdifferentials.
Appendix~\ref{App:relation} provides proofs for the results in Subsection~\ref{subsec:notion}. Appendix~\ref{IRRF} presents results about an important recursive formula.
Appendix~\ref{APP:analysis} presents some useful technical results used in Section~\ref{sec:proof}.
Finally, Appendix~\ref{sec:termination} presents a complementary convergence guarantee in terms of the Moreau envelope using PBF.

    \section{Background and Related Methods}
    \label{Sec:background}

This section provides the background material used throughout the paper and reviews the PS method for weakly convex optimization.
Subsection~\ref{subsec:DefNot} introduces the basic definitions and notation.
Subsection~\ref{subsec:notion} presents the stationarity notion considered in this paper and two other related notions and describes their relationships.
Finally, Subsection~\ref{SEC:SCS} reviews the deterministic version of the stochastic proximal subgradient method of~\cite{davis2019stochastic} and briefly discusses the relation between PBF and bundle methods for DC programming.

\subsection{Basic definitions and notation} \label{subsec:DefNot}

    The sets of real numbers and positive real numbers are denoted by $\R$ and $\R_{++}$, respectively. 
    % Let $\R$ denote the set of real numbers.
    % Let $ \R_+ $ and $ \R_{++} $ denote the set of non-negative real numbers and the set of positive real numbers, respectively.
	Let $\R^n$ denote the standard $n$-dimensional Euclidean space equipped with  inner product and norm denoted by $\left\langle \cdot,\cdot\right\rangle $
	and $\|\cdot\|$, respectively. 
% 	Given a set $ S\subset \R^n $, its linear (resp., convex) hull is
% 	denoted by $\text{Lin} \, S $ (resp., ${\rm conv} \, S$).
	Let $\log(\cdot)$ denote the natural logarithm and $\log^+(\cdot)$ denote $\max\{\log(\cdot),0\}$. Let ${\cal O}$ denote the standard big-O notation.

	For a given function $\varphi: \R^n\rightarrow (-\infty,+\infty]$, let $\dom \varphi:=\{x \in \R^n: \varphi (x) <\infty\}$ denote the effective domain of $\varphi$ 
	and $\varphi$ is proper if $\dom \varphi \ne \emptyset$.
	A proper function $\varphi: \R^n\rightarrow (-\infty,+\infty]$ is $\mu$-convex for some $\mu \ge 0$ if
	$$
	\varphi(\lam x+(1-\lam) y)\leq \lam \varphi(x)+(1-\lam)\varphi(y) - \frac{\mu \lam (1-\lam)}{2}\|x-y\|^2
	$$
	for every $x, y \in \dom \varphi$ and $\lam \in [0,1]$.
 Denote the set of all proper lower semi-continuous convex functions by $\bConv{n}$.
 % When $\mu=0$, we simply denote
 %    $\mConv{n}$ by $\bConv{n}$.
	For $\varepsilon \ge 0$, the \emph{$\varepsilon$-subdifferential} of $ \varphi $ at $x \in \dom \varphi$ is denoted by
	\beq \label{def:subdif}
 \partial_\varepsilon \varphi (x):=\left\{ s \in\R^n: \varphi(y)\geq \varphi(x)+\left\langle s,y-x\right\rangle -\varepsilon, \forall y\in\R^n\right\}.
        \eeq
	For simplicity, the subdifferential of $\varphi$ at $x \in \dom \varphi$, i.e., $\partial_0 \varphi(x)$, is denoted by $\partial \varphi (x)$.
         The set of proper lower semi-continuous convex functions $\Gamma$ such that $\Gamma \le \varphi$ is denoted by ${\cal B}(\varphi)$ and any such $\Gamma$ is called a bundle for $\varphi$.

\subsection{Three notions of stationary points}\label{subsec:notion}

This subsection introduces the subdifferential and directional derivative of a general proper lower semi-continuous function. It then defines the regularized stationary point sought by the main algorithm of this paper. In addition, it presents two alternative stationarity conditions, formulated in terms of the directional derivative and the Moreau envelope, respectively, and clarifies the relationships among these three notions.

We start by giving the definitions of
directional derivative
of a proper lower semi-continuous  function.

\begin{definition}\label{def:di-deriv}
The directional derivative $\phi'(x;d)$
of
$\phi$ at $x$ along $d$ is
\[
\phi'(x;d) := \liminf _{t \downarrow 0} \frac{\phi(x+td)-\phi(x)}{t}.
\]
\end{definition}

% We start by giving the definition of
% the
% $\varepsilon$-subdifferential
% of a proper closed function.

% In \cite{kruger2003frechet}, the Frechet $\varepsilon$-subdifferential is defined in the following way:

% The next definition for $\varepsilon$-subdifferential can be found 
% in  Definition 1.10 of \cite{kruger2003frechet}.

% [Frechet $\varepsilon$-subdifferential]

Motivated by the first-order directional optimality condition, we introduce
the following  intrinsic  notion of approximate stationarity.

\begin{definition}[Directional stationary point] \label{def:dsp}
    For a pair $(\varepsilon_D,\delta_D) \in \R^2_{++}$, a point $x\in \dom \phi$ is called a $(\varepsilon_D,\delta_D)$-directional stationary point if there exists $\tilde x \in \dom \phi$ such that 
    \begin{equation}\label{eq:dir}
    \|x- \tilde x\|\le \delta_D, \quad \inf_{\|d\|\le 1} \phi'(\tilde x;d) \ge -\varepsilon_D.
    \end{equation}
\end{definition}
When $(\varepsilon_D,\delta_D)=(0,0)$, then \eqref{eq:dir}
reduces to the condition that $\phi'(x;d) \ge 0$ for all $d \in \R^n$, a condition which is known to be equivalent
to $0 \in \partial \phi(x)$, which denotes the Frechet subdifferential and formally defined in Definition~\ref{def:Fre_subdifferential} below.

\begin{definition}\label{def:Fre_subdifferential}
The Frechet subdifferential of a proper lower semi-continuous function
$\phi: \mathbb{R}^{n} \rightarrow \mathbb{R} \cup\{\infty\}$ is defined as 
\[
 \partial \phi(x)=\left\{v \in \R^n: \liminf _{y \rightarrow x} \frac{\phi(y)-\phi(x)-\left\langle v, y-x\right\rangle}{\|y-x\|} \geq 0 
\quad \forall x \in \R^n \right\}.
\]
% When $\varepsilon=0$, we denote
% $\partial_{0} \phi(x)$ simply by $\partial \phi(x)$.
\end{definition}

For every function $g(\cdot)$, $m \in \R_+$, $z \in \R^n$,
    let 
    \begin{equation}
   g_m(\cdot;z) := g (\cdot) + \frac{m}{2}\|\cdot - z\|^2.
    \label{def:phim}
    \end{equation}
It is well known (see, for example, Lemma 2.1 of \cite{davis2019stochastic}) that for an $m$-weakly convex function $\phi$,
\begin{equation}\label{eq:key_chara1}
    \partial \phi(x) 
 =\left\{ v \in \R^n :  \phi(y) \geq \phi(x)+\langle v, y-x\rangle-\frac{ m}{2}\|y-x\|^{2}, \ \forall y \in \R^n \right\} =
 \partial \left[\phi_{ m}(\cdot;x)
 \right](x).
 \end{equation}
    
% The following result  provides a characterization of the
% Frechet subdifferential 
% for a weakly convex function. 

% \begin{proposition}\label{prop:Fre_sub}
% \label{prop:Frechet relation}
% Assume that $\phi: \mathbb{R}^{
% n} \rightarrow \mathbb{R} \cup\{\infty\}$ is a closed $m$-weakly convex function. Then  we have
%  \begin{align}
%     \partial \phi(x) 
%  &=\left\{ v \in \R^n :  \phi(y) \geq \phi(x)+\langle v, y-x\rangle-\frac{ m}{2}\|y-x\|^{2}, \ \forall y \in \R^n \right\} \label{eq:key_chara} \\
%  &=
%  \partial \left[\phi_{ m}(\cdot;x)
%  \right](x)  \label{eq:key_chara1}.
%  \end{align}
% \end{proposition}
% \begin{proof}
% The proof of \eqref{eq:key_chara} can be found for example in Lemma 2.1 of \cite{davis2019stochastic}.
% Moreover, it follows from the definition of the subdifferential in \eqref{def:subdif} with $\varepsilon = 0$ and the definition of $\phi_m(\cdot;x)$ in \eqref{def:phim} that \eqref{eq:key_chara1} is equivalent to \eqref{eq:key_chara}.
% \end{proof}

% \subsection{Notions of approximate stationary points}
% \label{sec:notion}
% In this section, we define formally our new definition of approximate solution.

% This subsection introduces three notions of stationary points, including the one adopted by the main algorithm of this paper.
% It also describes how these notions are related to one another.

We now introduce the
definition of a regularized
stationary point of a
proper lower semi-continuous $m$-weakly convex function
$\phi$ which is the one
used by the main algorithm of this paper.
% This notion of regularized
% approximate stationary point generalizes the prox gradient mapping for analyzing the proximal method in a significant way so that it can be applied to analyze any algorithm for weakly convex nonsmooth problem.

% Assume for this subsection that
% $\phi$ is a
% closed $m$-weakly convex function.

\begin{definition}[Regularized stationary point]
\label{def:appr}
For a pair $(\bar\eta,\bar \varepsilon) \in \R^2_{++} $, a point $x\in \dom \phi$ is called a $(\bar \eta,\bar\varepsilon;m)$-regularized stationary point  of $\phi$
% problem \eqref{eq:opt_problem} 
if there exists a pair $(w,\varepsilon) \in \R^n \times \R_{++}$ such that 
\begin{equation}
  w \in  \partial_\varepsilon \left[\phi_{m}(\cdot;x) \right]  (x), \quad
\|w\| \le \bar \eta, \quad \ \varepsilon \le \bar \varepsilon.
\label{eq:app_sol}
\end{equation}
\end{definition}

We make two trivial remarks about the above definition. First,
if $(\bar \eta,\bar \varepsilon)=(0,0)$ then
it follows from 
\eqref{eq:key_chara1}
and Definition \ref{def:appr}
that
$x$ is a  $(\bar \eta,\bar \varepsilon;m)$-regularized stationary point if and only if
$x$ is an exact stationary point of \eqref{eq:opt_problem}, i.e.,
it satisfies $0 \in \partial \phi(x)$.
Second, when $m=0$ (and hence $\phi$ is convex),
the inclusion in
\eqref{eq:app_sol} reduces to $w \in \partial_{ \varepsilon} \phi(x)$, and Definition \ref{def:appr}
reduces to Definition 6.1 of \cite{liang2021proximal},
which has already been used in the analysis of proximal bundle methods for solving the convex version of \eqref{eq:opt_problem} (see Section 6 of \cite{liang2021proximal}).

The verification that $x$ is a $(\bar \eta,\bar\varepsilon;m)$-regularized stationary point
 requires exhibiting a pair of residuals
 $(w,\varepsilon)$ satisfying the inclusion in \eqref{eq:app_sol},
 which is generally not an immediate task. However, many algorithms for solving the convex version of \eqref{eq:opt_problem} and the one in this paper, are able to generate not only a sequence of iterates $\{x_k\}$ but also a sequence of corresponding residuals
 $\{(w_k,\varepsilon_k)\}$
 such that $(x,w,\varepsilon)=(x_k,w_k,\varepsilon_k)$
 satisfies the inclusion in \eqref{eq:app_sol},
 so that verification that
 $x_k$ is a $(\bar \eta,\bar\varepsilon;m)$-regularized stationary point simply amounts to checking the two inequalities in \eqref{eq:app_sol}.

% The following definition uses
% directional derivatives
% to define a different notion
% of a stationary point.

% Note that the following equality holds according to Proposition 8.32 in \cite{rockafellar2009variational}:
% \begin{equation}\label{eq:disdire}
% \operatorname{dist}(0 ; \partial \phi(\tilde x))=-\inf _{\|d\| \leq 1} \phi^{\prime}(\tilde x ; d).
% \end{equation}
% Thus a  $(\varepsilon_D,\delta_D)$-directional stationary point $\tilde x$ also satisifies
% \[
% \operatorname{dist}(0 ; \partial \phi(\tilde x)) \le \varepsilon_D.
% \]

Before stating the next notion of a stationary point based on the Moreau envelope, we introduce a slightly different
notation for the Moreau envelope which is more suitable for our presentation,
namely:
for any $\lam>0$ and $x \in \R^n$, let
\begin{equation}\label{eq:Moreau}
  \hat M^{\lam}(x) :=
\min_u \left\{ \phi(u) + \frac{\lam^{-1}+ m}{2 }\|u-x\|^2 \right\}.   
\end{equation}
Note that \eqref{eq:Moreau} is the Moreau envelope of $\phi_m$ and therefore depends on $m$. Since $m$ is fixed throughout the paper, we omit this dependence from the notation for simplicity.
 
 The  gradient formula for the Moreau envelope (see Section 1 in \cite{davis2019stochastic}) is as follows:
\begin{equation}\label{eq:new_grad}
     \nabla  \hat M^{\lam}(x) = \left(\frac{1}{\lam}+m\right) \left(x - \hat x^\lam(x) \right)
\end{equation}
where 
\begin{equation}\label{eq:hatx}
    \hat x^\lam(x) := \argmin_{u \in \R^n} \left\{\phi(u)+ \frac{\lam^{-1}+m}2\|u -  x\|^2\right\}.
\end{equation}
The Moreau envelope gradient provides a commonly used stationarity measure
for weakly convex optimization; see, for example,
\cite{davis2019stochastic,davis2019proximally}.
Indeed, \eqref{eq:new_grad} shows that
$\|\nabla \hat M^\lambda(x)\|$ is the scaled residual between $x$ and the
solution $\hat x^\lambda(x)$ of the corresponding proximal subproblem.
Moreover, the optimality condition for \eqref{eq:hatx} implies
\[
\dist\left(0,\partial\phi(\hat x^\lambda(x))\right)
\leq
\|\nabla \hat M^\lambda(x)\|.
\]
Hence, a small Moreau envelope gradient means that $x$ is close to a nearby point that is nearly stationary for $\phi$.

This interpretation motivates the following notion of stationarity based on the Moreau envelope.

\begin{definition}[Moreau stationary point]
    For any $\varepsilon_M >0$ and $\lam>0$, a point $x\in \dom \phi$ is called a $(\varepsilon_M;\lam)$-Moreau stationary point if 
    $\|\nabla \hat M^{\lam}(x)\| \le \varepsilon_M$.
\label{def:Moreau}
\end{definition}

It is worth noting that among the three notions of stationary points,
 the directional stationary one is the only one which does not depend on the weak convexity parameter $m$.

The next result,
whose proof is given in Appendix \ref{App:relation}, describes the relationship between directional stationary points and Moreau stationary points.

% The relationship of these three notions of stationary point is summarized in the following two results and the  proofs of these two results are
% in Appendix \ref{App:relation}.

\begin{proposition}
Assume $\lam >0$ and $\phi$ is a $m$-weakly convex function. Then, the following statements hold:
\begin{itemize}
    \item [a)]
    if $x$ is a $(\varepsilon_D,\delta_D)$-directional stationary point then $x$ is a $\left(\varepsilon_M;\lam \right)$-Moreau stationary point where
    \[
    \varepsilon_M = \left(m+\frac{1}{\lam}\right)\left[(3+2\lam m)\delta_D + 2\lam \varepsilon_D \right];
    \]
    \item [b)]
    if $x$ is a $(\varepsilon_M;\lam)$-Moreau stationary point then $x$ is a $(\varepsilon_M,\varepsilon_M/(m+\lam^{-1}))$-directional stationary point.
\end{itemize}
\label{prop:relation}
\end{proposition}

The following result, whose  proof is given in Appendix \ref{App:relation},
provides an equivalent characterization of
a directional stationary point in terms of the subdifferential of $\phi$.

\begin{proposition}\label{prop:subdir}
Assume that $\phi$ is an $m$-weakly convex function. Then,
        $x$ is a $(\varepsilon_D,\delta_D)$-directional stationary point if and only if  there exists $\tilde x \in \dom \phi$ such that 
        \[
  \|x- \tilde x\|\le \delta_D, \quad \operatorname{dist}(0; \partial \phi(\tilde x)) \le \varepsilon_D.
\]
\end{proposition}

The following result, whose  proof is given in Appendix \ref{App:relation},
shows that a
regularized stationary point is
both a 
directional stationary point and a Moreau stationary point.

\begin{proposition}\label{key:rela}
Assume that $\phi$ is an $m$-weakly convex function. If $x$ is a $(\bar \eta,\bar \varepsilon;m)$-regularized stationary point, then the following statements hold:
\begin{itemize}\label{prop:ours}
    \item[a)] $x$ is a $(\bar \eta + 2 \sqrt{2 m \bar \varepsilon }, \sqrt{2\bar \varepsilon/m})$-directional stationary point;
    \item[b)]
    $x$ is a $(18\sqrt{2m\bar \varepsilon}+4\bar \eta;1/m)$-Moreau stationary point.
\end{itemize}
    % \begin{itemize}
    %     \item[a)] 
    %     $y$ is a $(\bar \eta + 2 \sqrt{2 m \bar \varepsilon }, \sqrt{2\bar \varepsilon/m})$-directional stationary point;     
    %     \item[b)] $y$ is a $(2\sqrt{4m\bar \varepsilon+2\bar \eta^2};1/2m)$-Moreau stationary point.
    % \end{itemize}
\end{proposition}

\subsection{Related methods for weakly convex nonsmooth optimization}\label{SEC:SCS}

This subsection reviews two related approaches for weakly convex optimization. We first recall the deterministic version of the PS method of~\cite{davis2019stochastic}, and then briefly discuss the relation between PBF and bundle methods for DC programming.

More specifically, this subsection considers
the PS method described below under the assumptions stated below \eqref{eq:opt_problem}
except that the  constant $L$ is assumed to be zero.

\noindent\rule[0.5ex]{1\columnwidth}{1pt}
	
	PS

    \noindent\rule[0.5ex]{1\columnwidth}{1pt}
    {\bf Input:} $\hat x_0 \in \dom h$, a sequence of $\left\{\alpha_t\right\}_{t \geq 0} \subset \mathbb{R}_{+}$ and iteration count $T$.\\
    {\bf Step:} For $t = 0,1,\cdots, T$, compute
\begin{equation}
 \hat x_{t+1}=\argmin_{u \in \R^n} \left\{\ell_{f}(u;\hat x_t)+ h(u)+ \frac{1}{2\alpha_t}\|u -  \hat x_t\|^2\right\}.
\end{equation}
    \noindent\rule[0.5ex]{1\columnwidth}{1pt}

The following result (see Theorem 3.4 in \cite{davis2019stochastic}) describes the rate
of convergence of the PS method.

\begin{proposition}
\label{prop:Davis}
Assume $(f,h)$ are functions satisfying conditions below \eqref{eq:opt_problem}
with $L=0$ and
assume that  $\alpha_t \in(0,1 / \bar{m}]$ for every $t \ge 0$ for some
$\bar m \in (m,2m]$. Then,
the iterate $\hat x_t$ of
PS satisfies
\begin{equation}\label{Moreacomplexity}
\frac{\sum_{t=0}^T \alpha_t \left\|\nabla \hat M^{1 / (\bar{m}-m)}\left(\hat x_t\right)\right\|^2}{\sum_{t=0}^T \alpha_t}  \leq \frac{\bar{m}}{\bar{m}-m} \cdot \frac{\left(\hat M^{1 / (\bar{m}-m)}\left(\hat x_0\right)- \phi^*\right)+2 \bar{m} M^2 \sum_{t=0}^T \alpha_t^2}{\sum_{t=0}^T \alpha_t}.
\end{equation}
As a consequence, for any
given tolerance $\rho>0$ and constant $\gamma $ such that
$\gamma \in(0, 1/(2m)]$,
if the stepsizes $\alpha_t$ are chosen according to
$\alpha_t = \gamma/\sqrt{T+1}$
for every $t \ge 0$ and the iteration count $T$ satisfies
\begin{equation}\label{eq:MoreauT}
  T \ge \frac{\left[\left(\hat M^{1/m}\left(\hat x_0\right)- \phi^*\right)+4 m M^2 \gamma^2\right]^2}{\gamma^2 \rho^4},
\end{equation}
then
one of the iterates
$\hat x_t \in 
\{ \hat x_0, \ldots, \hat x_T\}$
% if the number of iterations $T$ 
of PS
% then one of the iterates
% $\hat x_t \in 
% \{ \hat x_0, \ldots, \hat x_T\}$
satisfies
$\left\|\nabla \hat M^{1 / m}\left(\hat x_t\right)\right\| \le \rho$.
\end{proposition}

% We now make some remarks
% about Algorithm ??? and Proposition \eqref{prop:Davis}.
% First,
% since $\nabla M_\phi^{1 / (2m)}\left(\hat x_t\right)$ is difficult to compute, a termination criterion
% based on this quantity is not
% computationally feasible.
% Instead, ??????
% Second,
% the method of
% \cite{davis2019stochastic}
% actually outputs
% an iterate $\hat x_{t^*}$ where $t^*$
% is
% sampled from $\{0, \ldots, T\}$ according to the probability mass function $\mathbb{P}\left(t^*=t\right)=1+m\lam_t/\sum_{i=0}^T 1+m\lam_i$ for every $t =0,\ldots,T$.
% It turns out that
% $\E[ \|\nabla M_\phi^{1 / (2m)}\left(\hat x_{t^*}\right)\|^2]$
% is equal to the left hand side of
% \eqref{Moreacomplexity} and hence is bounded
% by the right hand side of
% \eqref{Moreacomplexity}.
% However, a drawback of this approach is that it
% is not clear how
% to use $\hat x_{t^*}$ to
% develop an alternative stopping criterion
% that guarantees that
% $\hat x_{t^*}$ is a $(\rho,1/(2m))$-Moreau stationary point in expectation
% since neither the latter
% expected value nor
% the right hand side of \eqref{Moreacomplexity}
% is computable.

We now make some remarks
about PS and Proposition \ref{prop:Davis}.
First, PS always performs $T+1$ iterations
and the second part of Proposition \ref{prop:Davis} gives 
a sufficient condition on
$T$ which guarantees that
one of its iterates is a
$(\rho;1/m)$-Moreau stationary point.
An termination condition based on the magnitude
of $\nabla \hat M^{1 /m}(\hat x_t)$
is not doable since this quantity is generally expensive to compute.
Second, a drawback of the estimate
\eqref{eq:MoreauT} on $T$
is that it is
generally not computable as it depends on $\phi^*$.
Third,
the method of
\cite{davis2019stochastic}
actually outputs
an iterate $\hat x_{t^*}$ where $t^*$
is
sampled from $\{0, \ldots, T\}$ according to the probability mass function $\mathbb{P}(t^*=t)=\alpha_t/\sum_{i=0}^T \alpha_t$ for every $t =0,\ldots,T$.
It turns out that
$\E[ \|\nabla \hat M^{1 /m}(\hat x_{t^*})\|^2]$
is equal to the left-hand side of
\eqref{Moreacomplexity} and hence is bounded
by the right-hand side of
\eqref{Moreacomplexity}.
The advantage of this approach
is that, without
performing
any evaluation of $\|\nabla \hat M^{1 /m}(\cdot)\| ^2$,
it returns 
an iterate
$\hat x_{t^*}$
such that the expected
value of $\|\nabla \hat M^{1 /m}(\hat x_{t^*}) \|^2$
is bounded by the right-hand
side of \eqref{Moreacomplexity}. However, the authors are unaware of any
technique for de-randomizing this output strategy
due to the fact that
the function
$\|\nabla \hat M^{1 /m}(\cdot) \|^2$
 is generally nonconvex and hard to compute. 

 Additionally, the third remark in the Concluding Remarks discusses how the alternative inexact proximal point method developed in \cite{davis2019proximally} can
 estimate the gradient of the Moreau envelope evaluated at the center of each prox subproblem after it performs
 a pre-specified number of inner iterations to solve it approximately.

Finally, we conclude this subsection with a brief discussion of the relation
between PBF and bundle methods for DC programming. In particular, weak
convexity naturally induces a DC representation of \eqref{eq:opt_problem}. Indeed,
letting
\[
G(x)=f(x)+h(x)+\frac{m}{2}\|x\|^2,
\qquad
H(x)=\frac{m}{2}\|x\|^2,
\]
we have $\phi=G-H$, where both $G$ and $H$ are convex. Moreover, linearizing
$H$ at a point $y$ gives
\[
G(x)-H(y)-\langle \nabla H(y),x-y\rangle
=
\phi(x)+\frac{m}{2}\|x-y\|^2
=
\phi_m(x;y).
\]
Thus, PBF is closely related to proximal bundle methods for DC programming
and can be viewed as exploiting the particular quadratic DC decomposition
induced by weak convexity; see, e.g.,
\cite{de2019proximal,joki2017proximal,kanzow2024bundle}.
In contrast to general DC bundle methods, PBF does not require handling a
general nonsmooth second DC component: the latter is the quadratic function
$H$, which is linearized exactly. Our analysis exploits this structure to
derive explicit iteration-complexity bounds for obtaining regularized
stationary points, together with their connections to Moreau stationarity.

\section{Proximal Bundle Framework and Its Complexity}
\label{Sec:PBF}

This section contains three subsections. The first one describes the main problem and the assumptions made on it.
 The second one describes the  proximal bundle framework (PBF) and describes its main complexity result. The last one  presents two special instances of PBF.

\subsection{Problem description and main assumptions}
\label{sec:assumptions}

The main problem of this paper is described in \eqref{eq:opt_problem}
where the functions $f$ and $h$ are assumed to satisfy:
%\subsubsection{Main Assumptions}
\begin{itemize}
    \item[(A1)] functions $f,h:\mathbb{R}^{n} \rightarrow \mathbb{R} \cup\{\infty\}$ and
    a scalar $m>0$ such that
    $h \in \bConv{n}$,
    $f$ is $m$-weakly convex
    and $\dom h \subset \dom f$; (Hence, all complexity results in this paper are stated for $m>0$.)
		% $f_m(\cdot;0)$ and $h$ are are both in
		% $\bConv{n}$ and
		% $\dom h \subset \dom f$;
% 		\item[(A2)]
% 		the set of optimal solutions $X^*$ of
% 		problem \eqref{eq:opt_problem} is nonempty;

	\item[(A2)]
% 	$f:\mathbb{R}^{n} \rightarrow \mathbb{R} \cup\{\infty\}$ is a function such that $\dom f \supset \dom h$ and
		there exist constants
		$M, L \ge 0$ and
		a subgradient oracle,
		 i.e.,
		a function $f':\dom h \to \R^n$
		satisfying $f'(u) \in \partial f(u)$ for every $u \in \dom h$, such that
		\begin{equation}
		   	\|f'(u)-f'(v)\| \le 2M + L \|u-v\| \quad\forall u,v \in \dom h;
		\label{cond:A2} 
		\end{equation}
		
		\item[(A3)]
		$\phi^*:= \inf \{ \phi(u) : u \in \R^n \}$ is finite.
  \end{itemize}

  % {\bf ?????????????  Define $\ell_{f_{m}(\cdot;z)}(y;x)$ and $\ell_{\phi_{m}(\cdot;z)}(y;x)$}

	Let ${\cal C}(M,L)$ denote the class of functions $f$ satisfying Assumption (A2).
		Even though ${\cal C}(M,L)$ depends on $\dom h$, we have omitted this dependence from its notation. For any function
  $g \in {\cal C}(M,L)$ and $x \in \dom h$, we denote the linearization of $g$ at
  $x$ by
  \begin{equation}
\label{def:l_g}
  \ell_g(\cdot;x) := g(x) + \inner{g'(x)}{\cdot-x}.
  \end{equation}

To analyze the bundle models used by PBF, we need to control the error of
linearizing the convexified function $f_m(\cdot;z)$. The next lemma shows
that $f_m(\cdot;z)$ remains in the hybrid class ${\cal C}(M,L+m)$,
which leads to the linearization-error estimate in \eqref{eq:fm-linear} below.

	\begin{lemma}
 \label{lem:Lip_ineq}
Assume that $f\in {\cal C}(M,L)$ for some $(M,L) \in \R^2_+$ and $f$ is $m$-weakly convex
on $\dom h$.
Then, for every $z \in \R^n$ and $\tz \in \dom h$, we have
\[
f'(\tz) + m(\tz-z) \in \partial [ f_m(\cdot;z) ](\tz), \qquad
f_m(\cdot;z) \in 
\bConv{n} \cap {\cal C}(M,L+m).
\]
\end{lemma}

% \begin{proof}
% Since $f\in {\cal C}(M,L)$, there exists an oracle $f'$ satisfying
% the conditions in (A2), i.e.,
% $f'(x) \in \partial f(x)$ for every $x \in \dom f$ and
% \eqref{cond:A2} holds.
% Since $f$ is $m$-weakly convex, it follows from the definition of
% weak convexity that 
% $f_m \in \bConv{n}$. Moreover,
% it follows from Lemma \ref{lem:chara_weakly} with $(\varepsilon,x,c) = (0,\cdot,z)$ and
% the fact that
% $f'(x) \in \partial f(x)$ for every $x \in \dom f$
% that $f_m'(\cdot;z) := f'(\cdot) + m (\cdot - z) \in \partial [f_m(\cdot;z)](\cdot)$.
% The result now follows by noting that
% \eqref{cond:A2} and the definition
% of $f_m+1/\lam$ imply that
% for every $x,y \in \dom f$
% \begin{align*}
% \|f_m'(x;z) - f_m'(y;z)\| &= \|f'(x) - f'(y) + m(x-y)\| \\
% &\le \|f'(x) - f'(y)\| + m\|x - y\|\\
% &\le 2M + (L+m)\|x-y\|
% \end{align*}
% which means that $f_m(\cdot;z) \in {\cal C}(M,L+m)$.
% \end{proof}

\begin{proof}
  The convexity of $f_m(\cdot;z)$ follows from the $m$-weak convexity of $f$,
while the first inclusion follows from Lemma~\ref{lem:chara_weakly}. Moreover, \eqref{cond:A2}
implies that, for every $x,y\in\dom h$,
\[
\|f'_m(x;z)-f'_m(y;z)\| \le \|f'(x)-f'(y)\|+m\|x-y\| \le 2M+(L+m)\|x-y\|,
\]
and hence $f_m(\cdot;z)\in\mathcal{C}(M,L+m)$.
\end{proof}

In view of the first inclusion of
Lemma \ref{lem:Lip_ineq},
it follows that
for any $z \in \R^n$ and $\tz \in \dom h$, function $f_m(\cdot;z)$ admits the linearization
given by
\beq \label{def:linear}
% \ell_{f_m(\cdot;z)}(\cdot;\tz)
%\ell(\cdot;\tz,(f,z,m)):
\ell_{f_m(\cdot;z)} (\cdot;\tz) :=
 f_m(\tz;z) + \inner{f'(\tz)+m(\tz-z)\,}{\,\cdot-\tz}.
\eeq
Moreover, Lemma \ref{lem:Lip_ineq} implies that $f_m(\cdot;z)$ is convex and belongs to
$\mathcal{C}(M,L+m)$. Hence, the convexity of $f_m(\cdot;z)$ gives the
lower bound below, while the $(M,L+m)$-hybrid condition gives the upper
bound:
\beq \label{eq:fm-linear}
    0 \le   f_{m}(\cdot;z) - 
    \ell_{f_m(\cdot;z)} (\cdot;\tz)
    %\ell(\cdot;\tz,(f,z,m)) 
    \le    2 M \|\cdot-\tz\| + \frac{ L + 
    m}2 \|\cdot-\tz\|^2 \quad \forall z \in \R^n,\, \tz \in \dom h.
	\eeq
    Indeed, the second inequality in \eqref{eq:fm-linear} follows from a standard argument based on the fundamental theorem of calculus; see, for example, Lemma A.3 of \cite{liang2026proximal}.
    Note that the above observations with
 $\tz=z$ implies that for every $z \in \dom h$,
 we have
 \begin{equation}\label{eq:simlin}
   \ell_{f_m(\cdot;z)} (\cdot;z) = \ell_f(\cdot;z), \quad    
 \end{equation}
 and
 \beq \label{eq:fm-linear1}
   0 \le   f_{m}(\cdot;z) - 
    \ell_{f} (\cdot;z)
    %\ell(\cdot;\tz,(f,z,m)) 
    \le     2M \|\cdot-z\| + \frac{ L + 
    m}2 \|\cdot-z\|^2 \quad \forall z \in \dom h.
	\eeq
% Let $f: \R^n\rightarrow (-\infty,+\infty]$ be given. Assume f is $m$-weakly convex, the linearization of $f_m(\cdot;z)$ at $y \in \R^n$ is given by
        % \[

        % \]

 \subsection{The proximal bundle framework} \label{subsec:update}

 This subsection describes PBF and its main complexity result.
 It also compares the complexity of PBF with that of the deterministic version of the PS method of \cite{davis2019stochastic}
 described in Subsection~\ref{SEC:SCS}.

Before presenting  PBF, we first
provide a brief outline of the ideas behind it.
PBF consists of solving
a sequence of
subproblems of
the form 
\begin{equation}\label{def:subprob}
	    x_j = \underset{u\in \R^n}\argmin \left \{\Gamma_j (u) + \frac{1}{2\lam} \|u-x^c \|^2 \right\}
	\end{equation}
 where  $\Gamma_j$ is a relatively simple convex function 
 minorizing the convexification $\phi_m(\cdot;x^c)$ of $\phi$ 
 defined in \eqref{def:phim} and $x^c$ is a prox-center.
 As any classical proximal bundle approach, it updates the center $x^c$ during some iterations (called the serious ones) and keeps the center the same in the other ones (called the null ones).

% The statement of PBF uses the following definition.

% \begin{definition}\label{def:shadow}
%     Given $x^c \in \dom h$ and $\lam >0$, function $\bar{\Gamma}(\cdot) \in   {\cal B}(\phi_m(\cdot;x^c))$ is called a shadow of $\Gamma$ for \eqref{eq:x-pre} if it satisfies
% \[
% \bar{\Gamma}\left(x\right)=\Gamma\left(x\right), \quad 
% x = \underset{u\in \R^n}\argmin \left \{\bar \Gamma (u) + \frac{1}{2\lam} \|u-x^c\|^2 \right\},
% \]
% where $x$ is the optimal solution of \eqref{eq:x-pre}.
% \end{definition}

% \red{move this definition to Sec. 1.1 and give two simple examples: $\bar \Gamma = \Gamma$ and $\bar \Gamma = \ell_\Gamma(\cdot;x)$}

% {\bf Question:} Is it true that every mirrow $\bar \Gamma$ satisfies
% \[
% \bar \Gamma(\cdot) \ge \ell_\Gamma(\cdot;x)?
% \]

Following the description of PBF below, we also argue that it can be
viewed as
a specific implementation of
an inexact proximal point method
applied to \eqref{eq:opt_problem}.
The parameters
$m$, $ L$ and $M$  that appear on its description 
are as in Assumptions (A1) and (A2). We start by introducing the following notion of shadow function.

\begin{definition}\label{def:shadow}
    For any convex function $\varphi$, given $x^c \in \R ^n$, $\lam >0$, and $\Gamma(\cdot) \in   {\cal B}(\varphi)$, the function $\bar{\Gamma}(\cdot) \in   {\cal B}(\varphi)$ is called a shadow of $\Gamma$ for \eqref{eq:x-pre} if it satisfies
\[
\bar{\Gamma}\left(x\right)=\Gamma\left(x\right), \quad 
x = \underset{u\in \R^n}\argmin \left \{\bar \Gamma (u) + \frac{1}{2\lam} \|u-x^c\|^2 \right\},
\]
where $x$ is the optimal solution of \eqref{eq:x-pre}.
\end{definition}

% Obviously, for $\Gamma \in {\cal B}(\varphi)$, $\bar \Gamma = \Gamma$ and $\bar \Gamma(\cdot) = \Gamma(x) + \lam^{-1}\inner{x^c-x}{\cdot -x}$ are both shadows of $\Gamma$ for \eqref{eq:x-pre}.
% Moreover, for any functions $\Gamma \in {\cal B}(\varphi)$ and $\bar \Gamma$ satisfying Definition \ref{def:shadow}, the following inequality holds
% \begin{equation}\label{ineq:obs-linear}
%     \bar \Gamma(\cdot) \ge \bar \Gamma(x) + \frac{1}{\lam}\inner{x^c-x}{\cdot-x} = \Gamma(x) + \frac{1}{\lam}\inner{x^c-x}{\cdot-x},
% \end{equation}
% where the right-hand side is a linearization of $\Gamma$ at $x$ with subgradient $(x^c-x)/\lam \in \partial \Gamma(x)$.

	\noindent\rule[0.5ex]{1\columnwidth}{1pt}
	
	PBF
	
	\noindent\rule[0.5ex]{1\columnwidth}{1pt}

     \begin{itemize}
		\item [0.] 
    Let an initial point $ \hat x_0\in \dom h $, a tolerance pair   $(\bar \eta, \bar \varepsilon) \in \R^2_{++}$, 
    and a prox stepsize $\lambda>0$
		be given, and set  
		$y_0 = \hat y_0=\hat x_0$, $j=1$, $k=1$,  and
  \beq \label{eq:def-tildem}
  \delta= \min\left\{\frac{\bar \varepsilon}{16},\frac{\lam \bar \eta^2}{64(m\lam+2)}\right\}, \qquad \Gamma_1(\cdot) = \ell_{ f}(\cdot;\hat x_{0}) + h(\cdot);
  % \qquad m = m+ \frac{1}\lam;
  \eeq

    \item[1.] Compute the optimal solution $x_j$ and optimal value $\theta_j$ of \eqref{def:subprob} with $x^c = \hat y_{k-1}$;  
		if 
  \begin{equation}\label{def:txj}
 \phi_{m}(x_j;\hat y_{k-1}) + \frac{1}{2\lam} \| x_j -\hat y_{k-1}\|^2 
   < \phi_{ m}(y_{j-1};\hat y_{k-1}) + \frac{1}{2\lam} \| y_{j-1} -\hat y_{k-1}\|^2,	\end{equation}
   then set $y_j=x_j$; else, set $y_j = y_{j-1}$;
    \item [2.]
    Compute 
   \begin{equation}\label{def:wj}
       w_j :=\frac{1}{\lam}(\hat y_{k-1}-x_j) - m( y_j - \hat y_{k-1}), \quad \delta_j:=\delta+ \frac{\lam}{8(m \lam +1)}\| w_j\|^2;
   \end{equation}
		\item[3.] If
\begin{equation}\label{ineq:hpe1}
       t_{j} :=  \phi_{m}(y_j;\hat y_{k-1}) + \frac{1}{2\lam} \| y_j -\hat y_{k-1}\|^2     
       - \theta_{j} > \delta_j,
		\end{equation}
         then go to step 3a; else, go to step 3b;
         \begin{itemize}
         \item[3a)]
        (null update)
find
	    $\Gamma_{j+1}(\cdot) \in    {\cal B}(\phi_m(\cdot;\hat y_{k-1}))$ such that
        \begin{equation}\label{def:Gamma}
\Gamma_{j+1}(\cdot) = \max \left\{\ell_{f_m(\cdot;\hat y_{k-1})}(\cdot;x_j)+h(\cdot), \bar{\Gamma}_j(\cdot)\right\}
\end{equation}
     where $\bar \Gamma_j$ is a shadow of $\Gamma_j$ for \eqref{def:subprob} with $x^c = \hat y_{k-1}$ and go to step 4;
         \item[3b)] (serious update)
         set $\hat x_k=x_j$,
         $\hat y_k = y_j$, $\hat \Gamma_k(\cdot) = \Gamma_j(\cdot)$, $\hat w_k = w_j$, $\hat \delta_k = \delta_j$
         and
         \begin{equation}\label{def:vk}
\hat v_{k} = \frac{1}{\lambda} (\hat y_{k-1} - \hat x_{k}),
\end{equation}
   and      compute
         \begin{equation}\label{def:wk}
\hat \varepsilon_k =\phi_m (\hat y_{k};\hat y_{k-1}) -\hat \Gamma_{k}(\hat x_{k}) - \inner{\hat v_{k}}{\hat y_{k}-\hat x_{k}}
\end{equation}
 if $\|\hat w_k\| \le \bar \eta$ and
 $\hat \varepsilon_k \le \bar \varepsilon$, then stop;
 else find a bundle $\Gamma_{j+1}(\cdot) \in {\cal B}(\phi_m(\cdot;\hat y_k))$
         such that
    
    \begin{equation}
         % \ell_{\hat f_k}(\cdot;\hat x_{k}) 
   \Gamma_{j+1}(\cdot)  \ge \ell_f(\cdot;\hat y_k)+ h(\cdot),
\label{eq:serious}
    \end{equation}   
    where $\ell_f(\cdot;\cdot)$ is defined in \eqref{def:l_g}, set $k \leftarrow k+1$, and go to step 4;
        
\end{itemize}
% \item[3.]	compute
% \[
% w_k = ?????, \quad 
% \varepsilon_k = ???
% \]
%  if $\|w_k\| \le \bar \eta$ and
%  $\varepsilon_k \le \bar \varepsilon$, then stop;
%  else go step 4;
	    
	   % \in {\cal C}_{\phi_m(\hat y_{k+1};\hat x_k)}(\Gamma_j, x_{j}^c, \lambda, \sigma)$;

	\item[4.]	set $j\leftarrow j+1$, and go to step 1. 
  \end{itemize}
\rule[0.5ex]{1\columnwidth}{1pt}

% \[
% \phi_m(\hat y_k;\hat y_{k-1}) \ge \Gamma(\hat y_k)
% \]
% \[
%  \Gamma(\hat y_k) -\hat \Gamma_{k}(\hat x_{k}) - \inner{\hat v_{k}}{\hat y_{k}-\hat x_{k}} \ge 0
% \]

An iteration $j$ such that $t_j \le \delta_j$
     is called
	a serious iteration;
 % in which
	% case $x_j$ (resp., $y_j$) is called a serious iterate
	% (resp., auxiliary serious iterate????????);
	otherwise, $j$ is called a null iteration.
	% Moreover, we assume throughout our presentation that
	% $j=0$ is also a serious iteration.
 Let $
 % 0=j_0 \le 
 j_1 \le j_2 \le \ldots$ denote the sequence of all serious iterations and
 let $j_0:=0$.
 Define the $k$-th cycle ${\cal C}_k$ to be the iterations $j$ such that $j_{k-1}+1 \le j \le j_k$, i.e.,
 \begin{equation}\label{def:Ck}
     {\cal C}_k := \{ i_k,\ldots,j_k\}, \quad i_k := j_{k-1}+1.
 \end{equation}
 Hence, only the last 
 iteration of a cycle (which can be the first one
 if ${\cal C}_k$ contains only one iteration)  
 is serious.

 We now make some basic remarks about PBF.
 First, PBF is referred to as a framework since it does not completely
	specify the details of
how $\Gamma_{j+1}$ in either \eqref{def:Gamma} or \eqref{eq:serious} is updated.
    % Second, in view of \eqref{eq:serious} or the fact that the output of BUF satisfies \eqref{def:Gamma}, it follows that
    % % for every $j \ge 1$, we have
    % \begin{equation}\label{phi-property}
    %     \Gamma_j \le \phi_m (\cdot;\hat y_{k-1}),
    %     %????? := f_m(\cdot;\hat y_{k-1}) + h, 
    %     \qquad
    % \Gamma_j \in \bConv{n}, \qquad \forall j \in {\cal C}_k, \ k \ge 1.
    % \end{equation}
 Second, it is shown in
 Lemma~\ref{lemma:optimality} that $\hat w_k \in \partial_{\hat \varepsilon_k} \left[\phi_m(\cdot;\hat y_k) \right]  (\hat y_k)$ for every $k \ge 1$. Hence, $\hat y_k$ is 
 $(\bar \eta,\bar \varepsilon;m)$-regularized stationary point of $\phi$ whenever
 the stopping criterion
 $\|\hat w_k\| \le \bar \eta$ and
 $\hat \varepsilon_k \le \bar \varepsilon$ is satisfied
 in step~3b.
 Third, in view of the definition of $ y_j $ in \eqref{def:txj}
	and the above relation, it then follows that
	\begin{equation}\label{eq:minseq}
	 	y_j \in \Argmin \left\lbrace \phi_m(x;\hat y_{k-1}) +\frac{1}{2\lam}\|x - \hat y_{k-1}\|^2:
		x \in \{ \hat y_{k-1},x_{i_k},\ldots,x_j\}	\right\rbrace. 
	\end{equation}
In words, $x_j$ is the trial point obtained by minimizing the current bundle model, whereas $y_j$ is the best point encountered so far within the current cycle in terms of the regularized objective. Thus, $x_j$ is used to refine the model, while $y_j$ serves as the incumbent candidate for updating the prox-center.
Fourth, two simple ways of choosing the bundle function $\Gamma_{j+1}$ such that $\Gamma_{j+1}(\cdot) \in {\cal B}(\phi_m(\cdot;\hat y_k))$ and \eqref{eq:serious} holds
 are to set
$\Gamma_{j+1}=\ell_{f} (\cdot;\hat y_{k}) + h$ or $\Gamma_{j+1}=\max\{\tilde \Gamma_{j+1}, \ell_{f}(\cdot;\hat y_{k}) +h\}$, where
 \[
 \tilde \Gamma_{j+1} := \Gamma_j - m \inner{\hat y_k -\hat y_{k-1}}{\cdot-\hat y_{k}} - \frac{m}2 \|\hat y_{k}-\hat y_{k-1}\|^2.
 \]
Clearly, the first choice for $\Gamma_{j+1}$ satisfies both conditions and the second one
 satisfies \eqref{eq:serious}.
 Moreover, using the fact that $\Gamma_j \le \phi_m(\cdot;\hat y_{k-1})$ and the definition of
 $\phi_m(\cdot;\hat y_{k-1})$ in \eqref{def:phim}, it is easy to see that
 $\tilde \Gamma_{j+1} \le \phi_m(\cdot;\hat y_{k})$, and hence that $\Gamma_{j+1} \le \phi_m(\cdot;\hat y_{k})$. 
 Fifth, 
PBF still requires knowledge of a parameter $m$ as in Assumption (A1),
and hence is not a completely universal method for finding a stationary point of \eqref{eq:opt_problem}.
Finally,  when $f$ is convex and $h\equiv 0$, the bundle update in \eqref{def:Gamma} is a special case of the one proposed in relations (4.7)-(4.9) of \cite{correa1993convergence} (see also (2.2)-(2.4) of \cite{diaz2023optimal}). 

Traditional proximal bundle methods often use serious/null decision policies
based on sufficient descent conditions. A detailed discussion of these
classical descent conditions and their relaxed variants can be found in
Subsection~3.1 of \cite{liang2021proximal}; see in particular the discussion
following equation~(3.7) therein.

{\bf PPM Interpretation:} 
We now discuss how PBF can be interpreted as an inexact proximal
point method for solving \eqref{eq:opt_problem}.
First,
the iterations within the $k$-th cycle can
be interpreted as regularized cutting plane
iterations applied to
the prox subproblem
\begin{equation}
%\bar \theta_k
\hat M^{\lam}(\hat y_{k-1}) = \min \left \{ %\psi^k(x) := 
%\phi_m(x;\hat{x}_{k-1}) + 
\phi(x) + \frac{1}{2 } 
\left( \frac1\lam + m \right)
%(\lam^{-1}+m) 
\|x - \hat y_{k-1}\|^2 : x \in \R^n \right \}.
\label{sub_pro}
\end{equation}

Second, the pair $(y_j,\Gamma_j)$ found in the serious (i.e., the last) iteration $j$ of the $k$-th cycle  approximately solves
\eqref{sub_pro}
% where $\lam'>0$ is such that
% $(\lam')^{-1}=\lam^{-1}+m$
according to \eqref{ineq:hpe1}, in which case
the point
$y_j$ as in step 1 of PBF becomes the center
$\hat y_k$ for the next cycle.
More specifically,
we now discuss how the inexact criterion \eqref{ineq:hpe1}
can be interpreted in
terms of the prox subproblem \eqref{sub_pro}.
Since $\Gamma_{j} \le \phi_m(\cdot;\hat y_{k-1})$, it follows from the definitions of $\theta_{j}$, $\hat M^{\lam}(\hat y_{k-1})$, and $\phi_m(\cdot;\hat y_{k-1})$ in step 1 of PBF,  \eqref{sub_pro}, and \eqref{def:phim}, respectively,
that $\theta_{j} \le \hat M^{\lam}(\hat y_{k-1})$, and hence that
 % \[
 %  m_j \le m_j^* \le \phi(y_j)+\frac1{2\lam}\|y_j-x_j^c\|^2
 % \]
 \begin{align}\label{ineq:tilde-tj}
     0 &\le 
     %\phi_m(y_{j};\hat y_{k-1})
     \phi(y_j) 
     +\frac12 \left(\frac1{\lam}+m\right)\|y_{j}-\hat y_{k-1}\|^2 - \hat M^{\lam}(\hat y_{k-1}) \nn \\
     &\le 
     %\phi_m(y_{j};\hat y_{k-1})
     \phi(y_j) 
     +\frac12\left(\frac1{\lam}+m\right)\|y_{j}-\hat y_{k-1}\|^2 - \theta_{j}
 \stackrel{\eqref{ineq:hpe1}} = t_{j}.
  % =: \tilde t_j
 \end{align}
Hence, if $j$ is an iteration 
for which 
\eqref{ineq:hpe1} does not hold  (i.e., a serious one), then it follows from \eqref{ineq:tilde-tj} that
$y_j$ is a $\delta_j$-solution of \eqref{sub_pro}. 

We now state the main complexity result for PBF. Its complexity bound
depends on the quantity
\begin{equation}\label{def:bar tau}
 \bar \sigma :=   \frac{\left[4\lam(L+m)+1\right]\left( M^2 + \beta_2 \left\{\beta_1[\hat M^{\lam}(\hat x_0) - \phi^* + 2\delta] + 4\zeta \lam M^2 \right\} \right)}{8\lam M^2}
\end{equation}
where
\begin{align}\label{beta}
  \beta_1 := 
  \left(m + \frac{2}{\zeta \lam}\right)\left(m+\frac{1}{\lam} \right)^{-1}, \quad
  \beta_2 := \left( \frac{ L+m}{2}+1\right) 
  \zeta^{-2}\left(\frac{1}{4\zeta \lam}+\frac{m}{2} \right)^{-1},
\end{align}
and
\begin{equation}\label{label:def_zeta}
      \zeta := \left\{ \begin{array}{ll}
         \frac{1}{2(L+m)\lam} \quad & \mbox{if $\lam > \frac{1}{2(L+m)}$};\\
        1 & \mbox{if $\lam \le \frac{1}{2(L+m)}$}.\end{array} \right. 
        \end{equation}

% \begin{theorem}[Main Theorem] \label{thm:main1}
% Given $(\bar \eta,\bar \varepsilon) \in \R^2_{++}$,
% PBF  with $\delta$ chosen as
% \begin{equation}\label{eq:def-tildem}
% \delta =\delta(\bar \eta,\bar \varepsilon) := \min\left\{\frac{\bar \varepsilon }{2},\frac{\lam \bar \eta^2}{8(m^2\lam^2+ 4\lam m + 4)}\right\}
% \end{equation}
% generates 
% in at most	
% \begin{equation}\label{cmplx:total}
% 		{\cal O}_1 \left( \left [\frac{\lam M^2}{\delta} + \lam (L+m) \log^+\left( \bar \sigma \right) \right] \frac{  (\hat M^\lam(\hat x_0) - \phi^*)(\lam m+1)}{\lam \bar \eta^2}\right)
% \end{equation}
% total iterations an iterate within the sequence $\{\hat y_k\}$ (as in step 3b of PBF)
% which is 
% a $(\bar \eta,\bar\varepsilon;m)$-regularized stationary point
%  of~$\phi$.
% \end{theorem}

% \begin{proposition}\label{thm:outer}
%  For a given tolerance pair $(\bar\eta,\bar \varepsilon) \in \R^2_{++}$, define
% \begin{equation}\label{def:K}
% K = K(\bar \eta,\bar\varepsilon) := \left\lceil \frac{8(m \lam +1)}{\lam \bar \eta^2} [\hat M^{\lam}(\hat x_0) - \phi^*] \right\rceil
% \end{equation}
%  where $\hat M^\lam(\cdot)$ is as in \eqref{eq:Moreau}.
% Then, PBF with $\delta$
% as in \eqref{eq:def-tildem}
% generates an iteration
% index $k \le K(\bar \eta,\bar \varepsilon)$
% such that
% \begin{equation}\label{ineq:termination}
%     \|\hat w_k\| \le \bar \eta,
% \quad \hat \varepsilon_k \le
% \bar \varepsilon.
% \end{equation}
% As a consequence,
% $\hat y_k$ is 
% a $(\bar \eta,\bar\varepsilon;m)$-regularized stationary point  of problem \eqref{eq:opt_problem}.
% \end{proposition}

\begin{theorem}[Main Theorem] \label{thm:main1}
% Given $(\bar \eta,\bar \varepsilon) \in \R^2_{++}$,
PBF 
generates 
a $(\bar \eta,\bar\varepsilon;m)$-regularized stationary point
 of~$\phi$
within the iteration sequence $\{\hat y_k\}$ (see step 3b of PBF)
in	
\begin{equation}\label{cmplx:total}
		 \left [4+ 4\lam (L+m) + \frac{16\lam M^2}{\delta} + [8\lam (L+m)+3] \log^+ \bar \sigma\right] \left[ 2\left(\frac{1}{\bar \varepsilon} + \frac{8(m \lam +1)}{\lam \bar \eta^2} \right)(\hat M^{\lam}(\hat x_0) - \phi^*) +2 \right]
\end{equation}
% \[
% \lam = c/m
% \]
% \begin{equation}\label{cmplx:total}
% 		{\cal O} \left( \left [\frac{M^2}{\delta m} +  \frac{L}{m} \log^+\left( \bar \sigma \right) + 1 \right] \left[\frac{  m\left[ \hat M^\lam(\hat x_0) - \phi^* \right]}{\bar \eta^2}+1 \right]\right)  \quad (??? Renato????)
% \end{equation}
total iterations where $(\bar \eta,\bar \varepsilon)$ is as in step 0 of PBF and $\delta$ is as in \eqref{eq:def-tildem}.
% total iterations an iterate within the sequence $\{\hat y_k\}$ (as in step 3b of PBF)
% which is 
% a $(\bar \eta,\bar\varepsilon;m)$-regularized stationary point
%  of~$\phi$.
\end{theorem}

In terms of the tolerances
$\bar \eta$ and $\bar \varepsilon$ only, it follows
from Theorem \ref{thm:main1} and the definition of $\delta$ in \eqref{eq:def-tildem} that the
iteration complexity of
PBF to find a $(\bar \eta,\bar\varepsilon;m)$-regularized stationary point of $\phi$ is
 ${\cal O}(\max \{ \bar \eta^{-2},\bar \varepsilon^{-1}\}^2 )$.
In particular, by choosing
$\lambda=1/[2(L+m)]$, we have $\zeta=1$, $\beta_1$ and $\beta_2$
remain bounded as $L$ increases, and
\[
\delta=\Theta\left(
\min\left\{\bar\varepsilon,
\frac{\bar\eta^2}{L+m}\right\}\right).
\]
Hence, up to logarithmic factors, the complexity bound in
Theorem~\ref{thm:main1} becomes
\begin{equation}\label{cmplx:total''}
   \tilde {\cal O}\left( \frac{ M^2 [\hat M^{\lam}(\hat x_0) - \phi^*]}{L+m} \max\left\{\frac{L + m}{\bar \eta^2 },\frac{ 1}{\bar \varepsilon}\right\}^2\right),
\end{equation}
which grows linearly with $L+m$ when the remaining problem parameters and tolerances are fixed.

% For the sake of stating an iteration-complexity for
% PBF to find a $(\rho;1/m)$-Moreau stationary point,
% we consider the specific
% case of PBF and Theorem \ref{thm:main1}
% with 
% parameters  $\lam$
% chosen as $\lam = \Theta(m^{-1})$ and $L =0$.
% \blue{In this case, it follows from \eqref{eq:def-tildem} that
% \begin{equation}\label{eq:simdelta}
% \delta = \Theta\left( \min \left \{ \frac{\bar \eta^2}{m},  \bar \varepsilon \right \}  \right).
% \end{equation}
% Thus, the complexity bound
% \eqref{cmplx:total} becomes
% \begin{equation}\label{cmplx:total'}
%    {\cal O}  \left(  \frac{M^2[\hat M^{\lam}(\hat x_0) - \phi^*]}{m \delta} \max\left\{\frac{m }{\bar \eta^2 },\frac{ 1}{\bar \varepsilon}\right\} \right)
% \end{equation}
% which, in view of \eqref{eq:def-tildem} with $L=0$, is
% }

The following result states the iteration complexity
for PBF to find an approximate Moreau stationary point. 
\begin{corollary}
For a given tolerance $\rho>0$, PBF with  $\lam = 1/m$ and 
$(\bar \eta,\bar \varepsilon)$ given by
\[
\bar \eta = \frac{\rho}8, \quad \bar \varepsilon=  \frac{\rho^2}{2592m}, 
\]
applied to an instance of \eqref{eq:opt_problem}
satisfying \eqref{cond:A2} with $L=0$
generates 
a  $(\rho;1/m)$-Moreau stationary point
 of~$\phi$
within the iteration sequence $\{\hat y_k\}$
in	
% finds in
\begin{equation}\label{new_comp}
    % \tilde {\cal O}_1 \left(  \frac{(\phi(\hat x_0) - \phi^*)M^2m}{\rho^4} + \frac{L(\phi(\hat x_0)- \phi^*)}{\rho^2} \right)
    {\cal O} \left(  \frac{M^2 m [\hat M^{\lam}(\hat x_0) - \phi^*]}{\rho^4}  \right)  
\end{equation}
 total iterations.
% an iterate within the sequence $\{\hat y_k\}$
% which is a 
% $(\rho;1/m)$-Moreau stationary point of $\phi$.
% \begin{equation}\label{eq:Moresta}
% \left\|\nabla M_\phi^{1 / (2m)}\left(x\right)\right\| \le \rho
% \end{equation}
\end{corollary}

\begin{proof}
We first note that the choice of $(\bar \eta, \bar \varepsilon)$ shows that PBF generates a $(\rho;1/m)$-Moreau stationary point of $\phi$ in view of Proposition \ref{key:rela}(b).
Therefore, the complexity bound \eqref{new_comp} immediately follows from \eqref{cmplx:total''} with $L=0$ and the choice of $(\bar \eta, \bar \varepsilon)$ in the corollary.
\end{proof}

We conclude this subsection by comparing PBF with the PS method of Subsection~\ref{SEC:SCS} in terms of both iteration complexity and prox stepsize selection.
First, it is worth noting that the iteration complexity of the PS method in
Subsection~\ref{SEC:SCS} is no better than that of PBF with
$\lam=1/m$ and $L=0$, namely, the bound in \eqref{new_comp}.
Indeed, the PS complexity bound in \eqref{eq:MoreauT}, viewed as a
function of $\gamma$, is minimized when
\[
\gamma
=
\Theta\left(
\frac{1}{M}
\sqrt{
\frac{\hat M^{\lam}(\hat x_0)-\phi^*}{m}
}
\right),
\]
in which case it matches the PBF bound \eqref{new_comp} up to constant
factors. However, this optimal choice of $\gamma$ is generally not
computable, since both $\hat M^{\lam}(\hat x_0)-\phi^*$ and $M$ may
be unknown.

Second, it is also worth noting that PBF uses the constant prox stepsize
\(\lambda=1/m\), whereas the PS method in Proposition~\ref{prop:Davis} uses
\[
\alpha_t=\frac{\gamma}{\sqrt{T+1}},
\qquad
\gamma\le\frac{1}{2m}.
\]
Thus,
\[
\alpha_t\le\frac{1}{2m\sqrt{T+1}},
\]
and hence the prox stepsize used by PBF is at least
$2\sqrt{T+1}$ times larger than that used by PS.

\subsection{Special instances of PBF}\label{sec:BUF}

Noting that the shadow $\bar \Gamma_j$ in step 3 of PBF is undetermined, PBF has the flexibility to choose $\bar \Gamma_j$. Next, we present two specific schemes of constructing $\bar \Gamma_j$, and hence two concrete ways to implement step 3a of PBF. 
The following two schemes are inspired by the proximal bundle methods for convex nonsmooth optimization \cite{liang2024unified,ruszczynski2011nonlinear,urruty1996convex}.

	\begin{itemize}
	   
	    \item \textbf{2-cut:} This scheme sets $\bar \Gamma_j = A_{j}+h(\cdot)$ where   $A_0(\cdot) = \ell_{f_m(\cdot;\hat y_0)}(\cdot;x_0)$ and  $A_j$ is recursively updated as follows
        \begin{equation}\label{def:Af+}
         A_{j+1}(\cdot) =  \tau A_j(\cdot) + (1-\tau) \ell_{f_m(\cdot;\hat y_{k-1})}(\cdot;x_{j-1})
     \end{equation}  for some $\tau \in [0,1]$.  In fact, \eqref{def:subprob} with $ \Gamma_j(\cdot)=\max \{A_j(\cdot),\ell_{f_m(\cdot;\hat y_{k-1})}(\cdot;x_{j-1})\}+h(\cdot)$ and $x^c = \hat y_{k-1}$  is equivalent to
        \begin{equation}\label{eq:xj-alt}
            x_j = \underset{u\in \R^n, t\in \R}\argmin \left \{t +h(u)+\frac{1}{2\lam}\|u-\hat y_{k-1}\|^2: \, A_j(u) \le t, \,\, \ell_{f_m(\cdot;\hat y_{k-1})}(u;x_{j-1}) \le t\right\},
        \end{equation}
         we denote by  $\tau$ an optimal Lagrange multiplier associated with the constraint  $A_j(u) \le t$. Moreover, it is easy to check from the optimality conditions of \eqref{eq:xj-alt} that $\bar \Gamma_j$ is a shadow of $\Gamma_j$ for \eqref{def:subprob} with $x^c = \hat y_{k-1}$.

	    \item  \textbf{multi-cut:}    This scheme sets $\bar \Gamma_j(\cdot) = \Gamma(\cdot;C(x_j))$ where 
      % computes $x$ as in \eqref{eq:x-pre}, 
     %  chooses the next bundle set $C_{j+1}$ so that
	    % \begin{equation}\label{eq:C+}
     %    C(x_j) \cup \{x_j\} \subset C_{j+1} \subset C_j \cup \{x_j\}
     %    \end{equation}
     %    where
     \begin{equation}\label{eq:Gamma-E3}
	        \Gamma(\cdot;C(x_j)) := \max \{  \ell_{f_m(\cdot;\hat y_{k-1})}(\cdot;c) : c \in C(x_j) \}+h(\cdot),
    \end{equation}   
        and   \begin{equation}\label{def:C+}
            C(x_j) := \{c \in C_j: \ell_{f_m(\cdot;\hat y_{k-1})}(x_j;c)+h(x_j) = \Gamma_j(x_j) \}.
        \end{equation}

         It can be shown that $\bar \Gamma_j$ is a shadow of $\Gamma_j$ for \eqref{def:subprob} with $x^c = \hat y_{k-1}$ following a similar argument as in the proof of Proposition D.2 of \cite{liang2024unified}. 
	\end{itemize}

We conclude this subsection with a few implementation remarks. Aside from the first-order oracle evaluation, the main computational cost of each PBF iteration is solving the regularized cutting-plane subproblem \eqref{def:subprob}. For the multi-cut scheme with $\ell$ cuts, \eqref{def:subprob} can be reformulated as a smooth concave maximization problem over the simplex. Each dual gradient evaluation requires one evaluation of the proximal mapping of $h$, and the dual problem can be solved using Nesterov's accelerated method. This results in $\mathcal{O}(\sqrt{\ell})$ evaluations of the proximal mapping of $h$. The 2-cut scheme admits a similar dual formulation, but with a one-dimensional dual variable. The resulting problem can be solved more simply by a bisection search, with each derivative evaluation again requiring one evaluation of the proximal mapping of $h$.

These observations highlight the practical tradeoff between the two schemes:
the 2-cut scheme keeps the model size fixed and is attractive when storage
and subproblem cost are primary concerns, whereas the multi-cut scheme may be
preferable when maintaining a larger bundle is computationally affordable.

\section{Proof of Theorem \ref{thm:main1}}
\label{sec:proof}

This section establishes the iteration complexity for obtaining the regularized stationary point stated in Theorem \ref{thm:main1}. The basic idea is to bound the total number of iterations by analyzing the number of cycles and the number of iterations within each cycle, a standard proof strategy in the analysis of proximal bundle methods; see, for example, \cite{fersztand2024modified,liang2021proximal,liang2024unified}.
More specifically, this section contains three subsections.
The first one derives a preliminary bound on the length of each cycle in terms of the tolerance $\delta$.
The second one bounds the number of
cycles generated by PBF with a specific choice
of $\delta$ until it obtains
a $(\bar \eta,\bar \varepsilon;m)$-regularized stationary point of $\phi$.
Finally, the third one derives the total iteration complexity of PBF with the aforementioned choice of $\delta$
until it obtains
a $(\bar \eta,\bar \varepsilon;m)$-regularized stationary point of $\phi$.

\subsection{Bounding the cardinalities of the cycles}\label{subsec:length}
This section
establishes a preliminary  upper bound on the cardinality of
each cycle ${\cal C}_k$
defined in \eqref{def:Ck}, and hence on
the number of null iterations
between two consecutive serious ones.

Throughout this section
and the next one, we let
\beq \label{eq:hatfk}
\hat f_{k}(\cdot) = f_m(\cdot;\hat y_{k-1}).
\eeq

	 The first result below presents a few basic properties of the null iterations between two consecutive serious ones.

	\begin{lemma}\label{lem:101}
	    For every $ j \in {\cal C}_k\setminus\{i_k\}$, there exists $\bar \Gamma_{j-1} \in \bConv{n} $ such that the following statements hold:
	 \begin{itemize}
	        \item[a)] 
	    \begin{align}
	            &\max \{ \bar \Gamma_{j-1} \,,\, \ell_{\hat f_{k}}(\cdot;x_{j-1})+h\} \le \Gamma_{j}, \label{eq:Gamma_j} \\
	            &\bar \Gamma_{j-1}(x_{j-1}) = \Gamma_{j-1}(x_{j-1}), \quad 
	        x_{j-1} = \underset{u\in \R^n}\argmin \left \{\bar \Gamma_{j-1} (u)  + \frac{1}{2\lambda} \|u-\hat y_{{k-1}}\|^2 \right\}; \label{eq:relation}
	    \end{align} 
	        \item[b)] for every
	    $u\in \dom h$, there holds
         \begin{equation}\label{eq:strcon}      
    \bar \Gamma_{j-1}(u)  + \frac{1}{2\lambda} \|u-\hat y_{k-1}\|^2 \ge \theta_{j-1}+ \frac{1}{2\lambda}\|u-x_{j-1}\|^2;
    \end{equation}
	    \end{itemize}
	\end{lemma}

	\begin{proof}
	    % a) This statement immediately follows from \eqref{def:subprob} with $x^c = \hat y_{k-1}$, step 3a of PBF, and Definition~\ref{def:shadow}.
        a) This statement follows directly from step 3a of PBF and
Definition~\ref{def:shadow}. In particular, the 2-cut and multi-cut schemes described
in Subsection~\ref{sec:BUF} provide two concrete constructions of such a shadow
function $\bar\Gamma_{j-1}$.

	    b) Using the second identity in \eqref{eq:relation}
     and the fact that
     $\bar \Gamma_{j-1}\in \bConv{n}$, we have for every $u\in\dom h$,
	    \[
	    \bar \Gamma_{j-1}(u)  + \frac{1}{2\lambda} \|u-\hat y_{k-1}\|^2  \ge \bar \Gamma_{j-1}(x_{j-1}) + \frac{1}{2\lambda} \|x_{j-1}- \hat y_{k-1}\|^2  + \frac{1}{2\lambda} \|u-x_{j-1}\|^2.
	    \]
	    The statement now follows from the above inequality, the first identity in \eqref{eq:relation}, and the definition of $\theta_{j-1}$ in step 1 of PBF.
	\end{proof}
	
The next lemma presents some basic facts about $\{\theta_j\}$ and $\{t_j\}$.

    \begin{lemma}\label{lem:theta_t}
        For every $ j \in {\cal C}_k\setminus\{i_k\}$, the following inequalities hold: 
   \begin{align}      
 \theta_{j} &\ge \theta_{j-1}+\frac{1}{2 \lam}\|x_{j}-x_{j-1}\|^2,\label{ineq:theta_j}\\
			t_{j} &+ \frac1{2\lam} \|x_{j}-x_{j-1}\|^2 \le t_{j-1}\label{ineq:t_j}.
   \end{align}
    \end{lemma}

    \begin{proof}
         Using the definition of $\theta_j$ in step 1 of PBF, \eqref{eq:Gamma_j}, and \eqref{eq:strcon} with $u=x_j$, we have 
    \begin{align*}
		\theta_{j}&=\Gamma_{j}(x_{j})+\frac{1}{2\lam}\|x_{j}-\hat y_{k-1}\|^2 \nn \\
		&\overset{\eqref{eq:Gamma_j}}\ge \bar \Gamma_{j-1}(x_{j})+\frac{1}{2\lam}\|x_{j}- \hat y_{k-1}\|^2
		\overset{\eqref{eq:strcon}}\ge \theta_{j-1}+\frac{1}{2 \lam}\|x_{j}-x_{j-1}\|^2 
	\end{align*}
    and thus \eqref{ineq:theta_j} holds.
 Using the definition of $ t_j $  in \eqref{ineq:hpe1}, \eqref{eq:minseq}, and \eqref{ineq:theta_j}, we have
		\begin{align*}
	 t_{j} & \stackrel{\eqref{ineq:hpe1}}=\phi_{ m}(y_{j};\hat y_{k-1}) + \frac{
  1}{2\lam} \|y_j - \hat y_{k-1}\|^2 - \theta_{j} \stackrel{\eqref{eq:minseq}}\le \phi_{ m}(y_{j-1};\hat y_{k-1})  + \frac{1}{2\lam} \|y_{j-1} - \hat y_{k-1}\|^2  - \theta_{j}  \\
            & \stackrel{\eqref{ineq:theta_j}}\le \phi_{m}(y_{j-1};\hat y_{k-1}) + \frac{1}{2\lam} \|y_{j-1} - \hat y_{k-1}\|^2  - \theta_{j-1} - \frac{1}{2\lam} \|x_{j} - x_{j-1}\|^2 \stackrel{\eqref{ineq:hpe1}} = t_{j-1} - \frac{1}{2\lam} \|x_{j} - x_{j-1}\|^2
		\end{align*}
        and thus \eqref{ineq:t_j} holds.
    \end{proof}

	% \textbf{Remark about one cut} \\
 %    Mono of tj Note that Lemma \ref{lem:theta_t}(a) does not apply to the one cut bundle update described in Subsection 3.1 of \cite{liang2024unified}.
    The following technical result provides an important recursive formula  for $\{t_j\}$.

% \begin{lemma}
% For every $j \in {\cal C}_k \setminus \{j_k\}$ there holds
%     \begin{equation}\label{eq:thetaj}
%     \theta_{j+1} \ge \theta_j+\frac{1}{2 \lam}\|x_{j+1}-x_j\|^2.
%     \end{equation}
% \end{lemma}

% \begin{proof}

% \end{proof}

 \begin{lemma}\label{lem:basic2}
 For every
  $ j \in {\cal C}_k\setminus\{i_k\}$, we have
			\[
			t_{j}\left(1 +  \frac{t_{j}}{  4\lam[2M^2 +   (L+m) t_{j}]  }\right)  \le t_{j-1}.
        \]
	\end{lemma}
	\begin{proof}		
    The key idea is to combine the decrease estimate in Lemma~\ref{lem:theta_t} with the
linearization error bound to obtain a quantitative lower bound on
$\|x_j-x_{j-1}\|$, which in turn yields the stated recursion for $t_j$.

  Using the definition of $t_j$ in \eqref{ineq:hpe1}, \eqref{eq:minseq}, and the definition of $\theta_j$ in step 1 of PBF, we conclude that 
		\begin{align*}
		t_j
  & \stackrel{\eqref{ineq:hpe1}}= 
  \phi_{ m}(y_j;\hat y_{k-1}) + \frac{1}{2\lam} \|y_j - \hat y_{k-1}\|^2  - \theta_j 
  \stackrel{  \eqref{eq:minseq}} \le
\phi_{ m}(x_j;\hat y_{k-1})   + \frac{1}{2\lam} \|x_j - \hat y_{k-1}\|^2  - \theta_j \\
&=
  \phi_{m} (x_j;\hat y_{k-1})    - \Gamma_j(x_j). 
  \end{align*}
  The above inequality,
  relations \eqref{eq:Gamma_j} and \eqref{eq:fm-linear}, and
  the fact that $\phi_{ m}(x_j;\hat y_{k-1}) = f_{ m}(x_j;\hat y_{k-1})+h(x_j)$, imply that
  \begin{align*}
0  &  \le \phi_m(x_j;\hat y_{k-1}) - (\ell_{\hat f_k}(x_j;x_{j-1})+h(x_j))-t_j\\
&\stackrel{\eqref{eq:Gamma_j}} = f_m(x_j;\hat y_{k-1}) - \ell_{\hat f_k}(x_j;x_{j-1})-t_j\\
		& \stackrel{\eqref{eq:fm-linear}}\le \frac{L+m}{2} \|x_j-x_{j-1}\|^2 + 2M \|x_j-x_{j-1}\|  -t_j,
		\end{align*}
% Rewrite the above inequality  as follows
% 		\[
% 		\frac{L+m}{2}  \| x_j-x_{j-1}\|^2 + 2M \|x_j-x_{j-1}\|- t_j \ge 0,
% 		\]
	which, together with $t_j \ge 0$ due to \eqref{ineq:tilde-tj},  can be easily seen to imply that 
    % $ \| x_j-x_{j-1}\| $ and using $t_j \ge 0$ in view of \eqref{ineq:tilde-tj}, we have
		\[
		\|x_j-x_{j-1}\| \ge \frac{-2M + \sqrt{4M^2 + 2 (L+m) t_j}}{L+m}= \frac{2 t_j}{ 2M + \sqrt{4M^2 + 2 (L+m) t_j}} \ge \frac{t_j}{\sqrt{4M^2 + 2 (L+m) t_j}}.
		\]
		The statement now follows from \eqref{ineq:t_j} and the above inequality.
	\end{proof}

	\begin{proposition}\label{prop:innercomp}
		For every cycle index $k \ge 1$ generated by PBF,
its size $|{\cal C}_k|$ is bounded by
    % $|{\cal C}_k| \le 2+ 4\lam (L+m) + \frac{16\lam M^2}{\delta} + [8\lam (L+m)+3] \log^+ \sigma_{k}$ for $l \in {\cal C}_k$ is 
		\begin{equation}\label{inner}
		N_k:=4+ 4\lam (L+m) + \frac{16\lam M^2}{\delta} + [8\lam (L+m)+3] \log^+ \sigma_{k}
		\end{equation}
  where
  \begin{equation}\label{def:tau} 
  \sigma_{k} := \frac{ t_{i_k}}{\delta +8\lam M^2/[4\lam (L+m)+1]}.
  \end{equation}
	\end{proposition}
	
	\begin{proof}	% It suffices to prove the number of iterations $l$ needed to find $t_l \le \delta$ is bounded by $N_k$ in view of the fact that $\hat \delta_k \ge \delta$. The proof of this claim is based on Lemma \ref{lem:keyrecur}. Note that  Lemma \ref{lem:basic2} implies that \eqref{ineq:recursive} is satisfied with $q_l = t_{l+i_k},a = 8\lam M^2$, and $b = 4\lam (L+m)$. Hence, Lemma \ref{lem:keyrecur} imply that there exists an  index $\hat l$ bounded by \eqref{inner} such that $t_{ \hat l+i_k} \le \delta$.
  %       ======================\\
        Suppose for contradiction that $|{\cal C}_k|> N_k$ and define $J =|{\cal C}_k|-2 $.
% This implies that
% there exists a nonnegative integer $J > 0$ such that  $J+1+i_k \in {\cal C}_k$ and $J+2> N_k$.
Since the $k$th cycle does not stop at iteration $j+i_k$ for any $j=0,\ldots,J$,
we have that
 \begin{equation}\label{vio2}
 q_j := t_{j+i_k}  > \delta  \quad \forall j = 0,\ldots,  J.
 \end{equation}
Observe that  Lemma \ref{lem:basic2} and \eqref{vio2} imply that assumptions \eqref{ineq:recursive1} and \eqref{ineq:recursive2} of Lemma \ref{lem:keyrecur} is satisfied with $a = 8\lam M^2$, $b = 4\lam (L+m)$, and $\{q_j\}_{j=0}^{J}$. Hence, the conclusion of Lemma \ref{lem:keyrecur} and the definition of $N_k$ in \eqref{inner} imply that  $J\le N_k-2$ which contradicts with the assumption $|{\cal C}_k|=J+2 > N_k$.
	\end{proof}

\subsection{Bounding total number of cycles}\label{subsec:cycle}

% \red{Question: check the outer analysis, when $m=0$, can we recover Theorem 6.4 of \cite{liang2021proximal}}

The goal of this subsection is to
establish a bound on
 the total number of cycles generated by PBF.

% Recall that $\hat x_k$, $\hat y_k$, $\hat \Gamma_k$ denote the
% last $x_j$, $y_j$ and $\Gamma_j$
% generated within a cycle, i.e., the ones with $j=j_k$.

The following technical result
provides the main properties of
the sequences $\{\hat x_k\}$, $\{\hat y_k\}$ and $\{\hat \Gamma_k\}$
generated by PBF.
Recall that
$\hat x_k$, $\hat y_k$, and $\hat \Gamma_k$ denote the
last $x_j$, $y_j$, and $\Gamma_j$, respectively,
generated within 
the cycle ${\cal C}_k$, i.e., the ones with $j=j_k$.
% which are sufficent
% to establish a bound on
% the total number of cycles.

\begin{lemma}\label{lem:basic_prop}
The  following statements hold
for every $k \ge 1$:
\begin{itemize}
    \item[a)] $\hat x_{k}$ is
 the optimal solution 
 % and
 %    optimal value 
    of
    \begin{equation}
        \label{eq:optmality}
        \min_{u \in \R^n} \left\{ \hat \Gamma_{k}(u)+\frac{1}{2\lambda} \| u - \hat y_{k-1}  \|^2 \right\};
        \end{equation}
        hence, if $\hat \theta_k$ denotes the optimal value of \eqref{eq:optmality}, then
        \begin{equation}\label{eq:hattheta}
\hat \theta_k = \hat \Gamma_k(\hat x_k) +\frac{1}{2\lambda} \| \hat x_k - \hat y_{k-1}  \|^2;
        \end{equation}
\item[b)] there hold
\begin{equation}
\hat \Gamma_{k}(\cdot) \in \bConv{n}, \quad \hat \Gamma_{k}(\cdot) \le \phi_m (\cdot;\hat y_{k-1})
\label{eq:basic}
\end{equation} and
        \begin{equation}
        \label{eq:control}
     \phi_m (\hat y_k;\hat y_{k-1})+\frac{1}{2\lam} \|\hat y_k - \hat y_{k-1}\|^2
        - \hat \theta_{k} \le \hat \delta_k.
        \end{equation}
\end{itemize}
\end{lemma}
\begin{proof}
a) The definition of $ \hat x_k $ in step 3b of PBF  and step 1 of PBF  imply that $\hat x_{k}$ is an  optimal solution
    of    \eqref{eq:optmality}.

b) 
 If $|C_k| = 1$, i.e., $j_k$ is also the first iteration of cycle $C_k$, then \eqref{eq:basic} follows from
 the fact that $\hat \Gamma_k(\cdot) = \Gamma_{i_k}(\cdot) \in {\cal B}(\phi_m(\cdot;\hat y_{k-1}))$. If $|C_k| > 1$, then 
$\hat \Gamma_k = \Gamma_{j_k}$ and hence satisfies \eqref{eq:basic} in view of step 3a of PBF.
Moreover, \eqref{eq:control} follows from the logic of the prox-center update rule in step 3 of PBF (see \eqref{ineq:hpe1}).
\end{proof}

The following result presents an important inclusion involving $(\hat y_k, \hat w_k, \hat \varepsilon_k)$
which implies that $\hat y_k$ is 
a $(\|\hat w_k\|, \hat \varepsilon_k;m)$-regularized stationary point of $\phi$.

\begin{lemma}
\label{lemma:optimality}
For every $k \ge 1 $, the quantities
$\hat x_k$, $\hat y_k$, $\hat v_k$, $\hat w_k$ and $\hat \varepsilon_k$ as in step 3b of PBF
satisfy
        \begin{equation}
           \hat \varepsilon_{k} \ge 0,  \qquad
           \hat v_{k} \in \partial_{\hat \varepsilon_{k}} [\phi_m (\cdot;\hat y_{k-1})] (\hat y_{k}),
           \label{eq:key_inclusion}
        \end{equation}
        \begin{equation}
              \hat w_{k} \in  \partial_{\hat \varepsilon_{k}} [\phi_m(\cdot;\hat y_{k}) ](\hat y_{k}).
              \label{eq:goal}
        \end{equation}
\end{lemma}
\begin{proof}
Since $\hat x_k$ is an optimal solution of \eqref{eq:optmality} in view of  Lemma \ref{lem:basic_prop}(a),
using
 the optimality condition 
 for \eqref{eq:optmality}, the fact that
 $\hat \Gamma_k \in \bConv{n}$, and the definition of $\hat v_k$ in \eqref{def:vk}, we have
$\hat v_{k} \in \partial \hat \Gamma_{k}(\hat x_{k})$.
This conclusion, \eqref{eq:basic}, the definition of subdifferential in \eqref{def:subdif}, 
and the definition of $\hat \varepsilon_{k}$ in \eqref{def:wk},
% in \eqref{def:var}, 
then imply that for every $u\in \dom h$,
\[
  \phi_m (u;\hat y_{k-1}) \ge  \hat \Gamma_{k}(u) \ge \hat \Gamma_{k}(\hat x_{k}) + \inner{\hat v_{k}}{u-\hat x_{k}} 
=\phi_m (\hat y_k;\hat y_{k-1}) + \inner{\hat v_{k}}{u-\hat y_{k}} - \hat \varepsilon_{k}
\]
% ----------
% \[
%  \sum \phi_m (u;\hat x_{i-1}) \ge \sum \phi_m (\hat y_i;\hat x_{i-1}) + \inner{\hat v_{i}}{u-\hat y_{i}} - \hat \varepsilon_{ i}
%  \]
%  ?????
%  \[
%  \sum \|u- x_i\|^2 = k \|u - x\|^2
%  \]
% ---------\\
% where the equality follows the definition
% of $\hat \varepsilon_{k}$ in \eqref{def:var}.
and hence that the inclusion in \eqref{eq:key_inclusion} holds.
The inequality in \eqref{eq:key_inclusion} follows from the above
inequality with $u=\hat y_{k}$.
Moreover,
the definition of $\hat w_{k}$ (see step 3b of PBF),
the inclusion in \eqref{eq:key_inclusion},
  and Lemma \ref{lem:chara_weakly} imply  that
\[
  \hat w_{k} \in \partial_{\hat \varepsilon_{k}} [\phi_m(\cdot;\hat y_{k-1}) ] (\hat y_{k})  - m(\hat y_{k} - \hat y_{k-1}) = \partial_{\hat \varepsilon_{k}} [\phi_m(\cdot;\hat y_{k}) ](\hat y_{k}),
\]
and hence that \eqref{eq:goal} holds.
\end{proof}

It follows from \eqref{eq:goal} that $\hat y_k$ is
a $(\|\hat w_k\|, \hat \varepsilon_k;m)$-regularized stationary point of $\phi$ where
the pair $(\hat w_k,\hat \varepsilon_k)$ can be easily computed according to step 3b of PBF.
Our remaining effort from now on will be to analyze the number of iterations it takes to obtain
an index $k$ such that
$\|\hat w_k\| \le \bar \eta$ and
$ \hat \varepsilon_k \le
\bar \varepsilon$,
and hence  
%the corresponding
% $\hat y_k$ is
a $(\bar \eta, \bar \varepsilon;m)$-regularized stationary point
$\hat y_k$ of $\phi$.

% Note that \eqref{eq:goal} is our key inclusion. The remaining work of this section is to bound $\hat \varepsilon_{k}$ and $\hat w_{k}$ .

The purpose of the following three
results is to establish a recursive
formula (see Lemma \ref{lem:key_recursive} below).
Lemma \ref{lemma:key_est} and Corollary \ref{cor:wk} 
are technical results that are needed
to prove Lemma \ref{lem:key_recursive}.
The next lemma bounds the stationarity residuals at a serious iteration in
terms of the displacement between two consecutive prox-centers. These
estimates will allow us to control the residuals using the cumulative
descent of the prox-centers.

\begin{lemma} \label{lemma:key_est}
 For every $k \ge 1$, the quantities $\hat x_k$, $\hat y_k$, $\hat w_k$ and $\hat \varepsilon_k$,
 as in step 3b of PBF, satisfy
    \begin{equation}
    \label{eq:key_est}
        \hat \varepsilon_{k} +
       \frac{1}{2\lambda}\|\hat y_{k} - \hat x_{k}\|^2  \le 2 \delta + \frac{1 + m\lam}{2\lam}\|
       \hat y_k - \hat y_{k-1}\|^2
    \end{equation}
    and
    \begin{equation}
    \|\hat w_{k} \|^2 
 \le \frac{8\delta}{\lam} +  4 \left( m+\frac{1}{\lambda} \right)^2\|\hat y_{k} - \hat y_{k-1}\|^2  
 \label{eq:est_w}
    \end{equation}
   where $\lam$ and $\delta$ are as in 
    step 0 of PBF.
\end{lemma}
%   \\  

% ------------------

% \[
% \frac12 w_k ^2 \le \frac 12 ( w^2_{k-1} - w_k^2) + \frac{4\delta}{\lam} +  2 \left[\left( m+\frac{1}{\lambda} \right)^2 + \frac{m\lam+1}{\lam^2} \right]\|\hat y_{k} - \hat y_{k-1}\|^2  
% \]
% Want
% \[
% \max \left \{ \frac{\|w_k\|^2}{\eta^2} , \frac{\varepsilon_k}{\bar \varepsilon} \right\} \le 1
% \]
% which is implied by
% \[
% \frac{\|w_k\|^2}{\eta^2} + \frac{\varepsilon_k}{\bar \varepsilon} \le 1
% \]

% =======================
% \[
% \|\hat w_{k} \|^2 
%  \le \frac{4\delta}{\lam} +  2 \left[\left( m+\frac{1}{\lambda} \right)^2 + \frac{m\lam+1}{\lam^2} \right]\|\hat y_{k} - \hat y_{k-1}\|^2  -\frac{4 \hat\varepsilon
%  _k}{\lam}
% \]

%   ==========================\\

% \begin{equation}\label{ineq:check}
%     \hat \varepsilon_k + C\left[\delta + \left[\frac{(m\lam+1)^2}{\lam} + \frac{m\lam+1}{\lam} \right]\|\hat y_{k} - \hat y_{k-1}\|^2\right] \ge 0
% \end{equation}
% % \[
% %     \hat \varepsilon_k \le - C\left[\delta - \left[\frac{(m\lam+1)^2}{\lam} + \frac{m\lam+1}{\lam} \right]\|\hat y_{k} - \hat y_{k-1}\|^2\right] 
% % \]
%  {\bf Strategy for adaptively changing $(m,\lam)$}

% If \eqref{ineq:check} is not satisfied, then 
% \[
% m = 1.2m
% \]

%  $\lam$ is updated as follows:
%  \[
%  \lam = \lam/1.5 \quad \mbox{if $I> 300$}
%  \]
%  \[
%  \lam = 1.1 \lam \quad \mbox{if $I< 5$}
%  \]
%  otherwise $\lam$ the same

\begin{proof}
Using both statements (a) and (b) of Lemma \ref{lem:basic_prop},
the definitions of $\hat \varepsilon_{k}$ and $\hat v_k$ in \eqref{def:wk} and \eqref{def:vk}, respectively, we
conclude that
% \begin{equation}
%         \label{eq:control}
%      \hat f_k(\hat y_{k};\hat y_{k-1}) +\frac{1}{2\lam} \|\hat y_k-\hat y_{k-1}\|^2 
%         - \hat \theta_{k} \le \delta.
%         \end{equation}
\begin{align}
    \hat \varepsilon_{k} &= \phi_m (\hat y_k;\hat y_{k-1})  - \hat \Gamma_{k}(\hat x_{k}) - \inner{\hat v_{k}}{\hat y_{k} - \hat x_{k}} \nonumber  \\
    & \stackrel{\eqref{eq:control}}\le\hat \delta_k  - \frac{1}{2\lam}\|\hat y_k-\hat y_{k-1}\|^2
    + \hat \theta_k - \hat \Gamma_{k}(\hat x_{k}) - \inner{\hat v_{k}}{\hat y_{k} - \hat x_{k}}  \nonumber \\
    &\stackrel{\eqref{eq:hattheta}}= \hat \delta_k - \frac{1}{2\lam}\|\hat y_k-\hat y_{k-1}\|^2
    + \frac{1}{2\lambda}\|\hat x_{k}-\hat y_{k-1}\|^2 - \inner{\hat v_{k}}{\hat y_{k} - \hat x_{k}}  \nonumber \\
    &\stackrel{\eqref{def:vk}}= \hat \delta_k - \frac{1 }{2\lam}\|\hat y_k-\hat y_{k-1}\|^2 
    + \frac{1}{2\lam} \|\hat x_{k}-\hat y_{k-1}\|^2 + \frac{1}{\lam} \inner{\hat x_{k}-\hat y_{k-1}}{\hat y_{k} - \hat x_{k}} \nonumber  \\
    % &= \hat \delta_k  - \frac{1 }{2\lam}\|\hat y_k-\hat y_{k-1}\|^2
    % + \frac{1}{2\lam} \left(
    % \|\hat y_{k} - \hat y_{k-1}\|^2 -
    % \| \hat y_{k} - \hat x_{k}\|^2
    %  \right) \nonumber\\
     & = \hat \delta_k  
    - \frac{1}{2\lam}  
    \| \hat y_{k} - \hat x_{k}\|^2. \label{est_var}
    \end{align}
    % and hence that \eqref{eq:key_est} holds. 
The above inequality together with  the definition of $\hat w_k$ and $\hat \delta_k$ in step 3b of PBF implies that
% \[
% \hat w_{k} = \hat v_{k} - m(\hat y_{k} - \hat y_{k-1}) = \frac{1}{\lambda}(\hat y_{k} - \hat x_{k}) - \left(m+\frac{1}{\lambda}\right)(\hat y_{k} - \hat y_{k-1})
% \]
% and hence that
\begin{align*}
    \|\hat w_{k} \|^2 &\stackrel{\eqref{def:wj}}\le \frac{2}{\lam^2}\|\hat y_{k} - \hat x_{k}\|^2 + 2 \left( m+\frac{1}{\lambda} \right)^2\|\hat y_{k} - \hat y_{k-1}\|^2 \\&
 \overset{\eqref{est_var}}\le \frac{4\hat \delta_k}{\lam}  + 2 \left( m+\frac{1}{\lambda} \right)^2\|\hat y_{k} - \hat y_{k-1}\|^2\\
 & \stackrel{\eqref{def:wj}}= \frac{4\delta}{\lam }+\frac{1}{2(m\lam+1)}\|\hat w_k\|^2+ 2 \left( m+\frac{1}{\lambda} \right)^2\|\hat y_{k} - \hat y_{k-1}\|^2
\end{align*}
 where the first inequality is due to the relation $\|a+b\|^2 \le 2 \|a\|^2 + 2 \|b\|^2$. Inequality \eqref{eq:est_w} now follows from the above inequality and simple algebraic manipulation. Using \eqref{eq:est_w}, \eqref{est_var} and the definition of $\hat \delta_k$ in step 3b of PBF, we have
 \begin{align*}
 \hat \varepsilon_{k} +
       \frac{1}{2\lambda}\|\hat y_{k} - \hat x_{k}\|^2 & \stackrel{\eqref{def:wj},\eqref{est_var}}\le  \delta + \frac{\lam}{8(m \lam+1)}\|\hat w_k\|^2\\
       &\stackrel{\eqref{eq:est_w}}\le \delta + \frac{\lam}{8(m\lam+1)}\left(\frac{8\delta}{\lam}+ 4 \left( m+\frac{1}{\lambda} \right)^2\|\hat y_{k} - \hat y_{k-1}\|^2    \right)\\
       &\le 2 \delta + \frac{1 + m\lam}{2\lam}\|
       \hat y_k - \hat y_{k-1}\|^2,
 \end{align*}
 and thus \eqref{eq:key_est} holds.
\end{proof}

\begin{corollary} \label{cor:wk}
    If for some $k \ge 1$, we have $\hat y_k = \hat y_{k-1}$, then
    the following statements hold:
    \begin{itemize}
        \item[a)]
        $y_j = \hat y_{k-1}$ for all $j \in {\cal C}_k$;
        \item[b)]
        $\|\hat w_k\| \le \bar \eta/4$ and $\hat \varepsilon_k \le \bar \varepsilon/8$, and hence $\hat y_k$ is 
a $(\bar \eta/4,\bar\varepsilon/8;m)$-regularized stationary point  of problem \eqref{eq:opt_problem}.
    \end{itemize}
\end{corollary}

\begin{proof}
a) Denote    
    $\Psi(x)=\phi_m(x;\hat y_{k-1}) +\frac{1}{2\lam}\|x - \hat y_{k-1}\|^2$. Assume that for contradiction that 
$j $ is the first index 
in ${\cal C}_k$ such that $y_j \ne \hat y_{k-1}$, and hence that $y_j \ne y_{j-1}=\hat y_{k-1}$. Because of the test \eqref{def:txj} in step 1 of PBF, it follows that
$y_j = x_j$ and
$\Psi(y_{j-1})> \Psi(x_j)$, and hence that  
\[
\Psi(\hat y_{k-1}) =
\Psi(y_{j-1})> \Psi(x_j) = \Psi(y_j) \stackrel{\eqref{eq:minseq}}\ge \Psi(\hat y_k) = \Psi(\hat y_{k-1}),
\]
which yields the desired contradiction.\\
b) Since $\hat y_k =\hat y_{k-1}$, \eqref{eq:est_w} implies
\begin{equation}\label{w}
\|\hat w_k\|^2 \le \frac{8\delta}{\lam} \le \frac{\bar \eta^2}{16} 
\end{equation}
where the last inequality is due the definition of $\delta$ in step 0 of PBF. 
% \red{Can the $\delta$ in step 0 of PBF be relaxed?}. 
Observe that \eqref{eq:key_est} implies that 
\begin{equation}\label{var}
\hat \varepsilon_k \le 2\delta \le \frac{\bar \varepsilon}{8}
\end{equation}
where the last inequality is due the definition of $\delta$ in step 0 of PBF.
Finally, \eqref{eq:goal}, \eqref{w}, \eqref{var}
and Definition~\ref{def:appr} then imply
that $\hat y_k$ is a
$(\bar \eta/4,\bar\varepsilon/8;m)$-regularized stationary point  of problem \eqref{eq:opt_problem}.  
\end{proof}

The above corollary indicates that whenever PBF has a repeated cycle which yields $\hat y_k = \hat y_{k-1}$ for some $k \ge 1$, the method will terminate and return a $(\bar \eta/4,\bar\varepsilon/8;m)$-regularized stationary point  of problem \eqref{eq:opt_problem}.

\begin{lemma}\label{lem:key_recursive}
For every  $k \ge 1$, the following statements are true:
 \begin{itemize}  \item[a)]there holds \begin{equation}\label{eq:bdd_diff}
  \phi(\hat y_k) + \frac12 \left( \frac1\lam + m \right)
  \|\hat y_k-\hat y_{k-1}\|^2 
  \le \phi(\hat y_{k-1})
\end{equation}
where
$\{\hat y_{k}\}$ is as in step 3b of PBF;
\item[b)] for every $k\ge 2$, there holds
\begin{equation}\label{eq:wvar_bdd}
    \frac{1}{2}\left(\frac{1}{\lam}+m\right)\sum_{l=2}^k \| \hat y_l - \hat y_{l-1}\|^2 \le \hat M^{\lam}(\hat x_0) -  \phi(\hat y_k)+2\delta.
\end{equation}
\end{itemize}
\end{lemma}

\begin{proof}
% For the sake of this proof, 
a) It follows from \eqref{def:phim} and \eqref{eq:minseq} that 
\begin{align*}
  \phi(\hat y_k) + \frac12 \left( \frac1\lam + m \right)
  \|\hat y_k-\hat y_{k-1}\|^2 
  \le  \phi(\hat y_{k-1}) + \frac12 \left( \frac1\lam + m \right)
  \|\hat y_{k-1}-\hat y_{k-1}\|^2  = \phi(\hat y_{k-1}),
    \end{align*}
 which proves \eqref{eq:bdd_diff}.  \\
 b) Summing \eqref{eq:bdd_diff} from $k = 2$ to $k$, we have
 \begin{equation}\label{sum2}
    \phi(\hat y_k) + \frac{1}{2}\left(\frac{1}{\lam}+m\right)\sum_{l=2}^k \|\hat y_l - \hat y_{l-1}\|^2 \le \phi(\hat y_1).
 \end{equation}
% Using \eqref{eq:control} with $k=1$, the definition of $\hat M^{\lam}(\hat x_0)$ in \eqref{eq:Moreau}, and the fact that $\hat \Gamma_1 (\cdot)\le \phi_m(\cdot;\hat x_0)$, we can conclude that
% \[
%  \phi(\hat y_1) + \frac{1}{2}\left(\frac{1}{\lam}+m\right) \|\hat y_1 - \hat y_{0}\|^2 \le \hat \theta_1 + \delta_1 \le \hat M^{\lam}(\hat x_0) +\delta.
Using \eqref{eq:control} with $k=1$, the definitions of $\hat M^{\lam}(\hat x_0)$ and $\theta_j$ in \eqref{eq:Moreau} and step 1 of PBF, respectively, and the fact that $\hat \Gamma_1 (\cdot)\le \phi_m(\cdot;\hat x_0)$, we can conclude that
\begin{align*}
 \phi(\hat y_1) + \frac{1}{2}\left(\frac{1}{\lam}+m\right) \|\hat y_1 - \hat y_{0}\|^2 &\stackrel{\eqref{eq:control}}\le \hat \theta_1 + \hat \delta_1 \stackrel{\eqref{eq:Moreau}}\le \hat M^{\lam}(\hat x_0) +\hat \delta_1 \\
 & \stackrel{\eqref{def:wj}}=  \hat M^{\lam}(\hat x_0)+ \delta+\frac{\lam}{8(m\lam+1)}\|\hat w_1\|^2, 
 \end{align*}
 where the identity is due to \eqref{def:wj}.
It thus follows from \eqref{eq:est_w} that
 \[
 \phi(\hat y_1) + \frac{1}{2}\left(\frac{1}{\lam}+m\right) \|\hat y_1 - \hat y_{0}\|^2 \overset{\eqref{eq:est_w}}\le \hat M^{\lam}(\hat x_0) +2\delta+ \frac{1}{2}\left(\frac{1}{\lam}+m\right) \|\hat y_1 - \hat y_{0}\|^2,
\]
which together with \eqref{sum2} implies that \eqref{eq:wvar_bdd}.
\end{proof}

We are now ready to bound the total
number of cycles generated by PBF.

\begin{proposition}\label{thm:outer}
 For a given tolerance pair $(\bar\eta,\bar \varepsilon) \in \R^2_{++}$, define
\begin{equation}\label{def:K}
K = K(\bar \eta,\bar\varepsilon) := \left\lceil 2\left(\frac{1}{\bar \varepsilon} + \frac{8(m \lam +1)}{\lam \bar \eta^2} \right)(\hat M^{\lam}(\hat x_0) - \phi^*) +1 \right\rceil
\end{equation}
 where $\hat M^\lam(\cdot)$ is as in \eqref{eq:Moreau}.
Then, PBF with $\delta$
as in \eqref{eq:def-tildem}
generates an iteration
index $k \le K(\bar \eta,\bar \varepsilon)$
such that
\begin{equation}\label{ineq:termination}
    \|\hat w_k\| \le \bar \eta,
\quad \hat \varepsilon_k \le
\bar \varepsilon.
\end{equation}
As a consequence,
$\hat y_k$ is 
a $(\bar \eta,\bar\varepsilon;m)$-regularized stationary point  of problem \eqref{eq:opt_problem}.
\end{proposition}

\begin{proof}
The proof is based on contradiction. If PBF does not terminate within the
first $K$ cycles, then each cycle has a sufficiently large normalized
stationarity residual. Summing these residuals and using the cumulative
prox-center displacement bound in Lemma~\ref{lem:key_recursive} leads to a contradiction.

Assuming by contradiction that PBF does not stop at the $K$-th cycle.
Then, the stopping criterion in step 3b of PBF implies that
\[
\max \left \{ \frac{\|\hat w_k\|^2}{\bar \eta^2}, \frac{\hat \varepsilon_k}{\bar \varepsilon} \right\} > 1 \quad \forall k=1,\ldots,K
\]
and thus 
\begin{equation}\label{vioooo}
\frac{\|\hat w_k\|^2}{\bar \eta^2} + \frac{\hat \varepsilon_k}{\bar \varepsilon} >1 \quad \forall k=1,\ldots,K.
\end{equation}
% Summing the above inequality from $k=1,\ldots,K$, we have
% \[
% \sum_{k=1}^K \left(\frac{\|\hat w_k\|^2}{\bar \eta^2} + \frac{\hat \varepsilon_k}{\bar \varepsilon}\right) > K,
% \]
Observe that the definitions of $\delta$ and $K$ in \eqref{eq:def-tildem} and \eqref{def:K}, respectively, imply that \begin{equation}\label{obsss}
    \left(\frac{4}{\bar \varepsilon} + \frac{16(m\lam+2)}{\lam \bar \eta^2} \right)\delta\le \frac{1}{2}, \quad \frac{1}{K-1}\left(\frac{1}{\bar \varepsilon}+ \frac{8(m\lam+1)}{\lam \bar \eta^2}\right)(  \hat M^{\lam}(\hat x_0) -  \phi^*) \le \frac{1}{2}.
\end{equation}
 Moreover, using \eqref{eq:key_est},   \eqref{eq:est_w} and \eqref{vioooo}, we have
\begin{align*}
1 &\overset{\eqref{vioooo}}< \frac{\|\hat w_k\|^2}{\bar \eta^2} + \frac{\hat \varepsilon_k}{\bar \varepsilon} \le  \frac{ 2\delta + (1 + m\lam)\|\hat y_{k} - \hat y_{k-1}\|^2/(2\lam)}{
       \bar \varepsilon
       }
       +\frac{
   8\delta +  4 (m \lam+1)^2\|\hat y_{k} - \hat y_{k-1}\|^2/\lam  }{\lam \bar \eta^2}\\
   & = \frac{2\delta}{\bar \varepsilon} + \frac{8\delta}{\lam \bar \eta^2}+  \left( \frac{  1 + m\lam}{
       2\lam \bar \varepsilon
       }
       +\frac{
    4 (m \lam+1)^2 }{\lam^2\bar \eta^2}\right) \|\hat y_{k} - \hat y_{k-1}\|^2.
    \end{align*}
    Summing the above inequality from $k=2$ to $K$, dividing the resulting inequality by $K-1$, and using \eqref{eq:wvar_bdd}, we have
    \begin{align*}
   1 &< \frac{2\delta}{\bar \varepsilon} + \frac{8\delta}{\lam \bar \eta^2}+  \left( \frac{  1 + m\lam}{
       2\lam \bar \varepsilon
       }
       +\frac{
    4 (m \lam+1)^2}{\lam^2\bar \eta^2}\right) \frac{1}{K-1}\sum_{k=2}^K \|\hat y_k - \hat y_{k-1}\|^2  \\
    &\overset{\eqref{eq:wvar_bdd}}\le \frac{2\delta}{\bar \varepsilon} + \frac{8\delta}{\lam \bar \eta^2}+  \left( \frac{  1 + m\lam}{
       2\lam \bar \varepsilon
       }
       +\frac{
    4 (m \lam+1)^2}{\lam^2\bar \eta^2}\right) \frac{2\lam}{(\lam m +1)(K-1)}(  \hat M^{\lam}(\hat x_0) -  \phi(\hat y_K) + 2\delta ) \\
    & \le \left(\frac{4}{\bar \varepsilon} + \frac{16(m\lam+2)}{\lam \bar \eta^2} \right)\delta+ \frac{1}{K-1}\left(\frac{1}{\bar \varepsilon}+ \frac{8(m\lam+1)}{\lam \bar \eta^2}\right)(  \hat M^{\lam}(\hat x_0) -  \phi^*)    \le 1
\end{align*}
where the second last inequality is due to the fact that $\phi(\hat y_K) \ge \phi^*$ and the last inequality is due  to \eqref{obsss}. Hence, the above inequality gives the desired contradiction and thus the statement is proved.
%  Finally, \eqref{eq:goal}, \eqref{eq:w}, \eqref{var2}
% and Definition \ref{def:appr} then imply
% that $\hat y_k$ is a
% $(\bar \eta,\bar\varepsilon;m)$-regularized stationary point  of problem \eqref{eq:opt_problem}.
\end{proof}

Before ending this subsection, we
observe that
the quantity 
$ \hat M^{\lam}(\hat x_0) - \phi^*$ in \eqref{def:K} can be majorized by
the more standard
initial primal gap $\phi(\hat x_0) -\phi^*$ due to
the definition
of $ \hat M^{\lam}(\cdot)$ in \eqref{eq:Moreau}.

\subsection{Proof of Theorem \ref{thm:main1}}\label{subsec:proof}

Recall that Proposition \ref{thm:outer} bounds the
total number of cycles while
Proposition \ref{prop:innercomp} provides a
bound on the cardinality of every
cycle ${\cal C}_k$ in terms of
$\sigma_{k}$. The following result
refines the latter result by
providing a uniform bound on
$\sigma_{k}$.

\begin{lemma}\label{lem:t1} 
      Define
\begin{equation}\label{def:bar t}
		    \bar t:=  M^2 + \beta_2 \left(\beta_1[\hat M^{\lam}(\hat x_0) - \phi^* + 2\delta] + 4\zeta \lam M^2 \right),
		\end{equation}
		where $\delta$ is as in \eqref{eq:def-tildem}, $\zeta$ is as in \eqref{label:def_zeta},  and $\beta_1$ and $\beta_2$ are as in \eqref{beta}. 
     Then,  the following statements hold for every $k \ge 1$:
     \begin{itemize}
         \item [a)]
         $ t_{i_k}\le \bar t$ where $t_{i_k}$ is as in \eqref{ineq:hpe1};
         \item[b)]
         $\sigma_{k} \le \bar \sigma$
         where $\sigma_k$ and $\bar \sigma$ are as in \eqref{def:tau} and \eqref{def:bar tau}, respectively.
     \end{itemize}	
	\end{lemma}

	\begin{proof}
a) 
	   	 It follows from \eqref{ineq:hpe1} with $j=i_k$ and the definition of $\theta_j$ in step 1 of PBF that
         \begin{align}
	        t_{i_k} & \stackrel{\eqref{ineq:hpe1}}= \phi_{m}(y_{i_k};\hat y_{k-1})  + \frac{1}{2\lam} \|y_{i_k} - \hat y_{k-1}\|^2 
	        -\theta_{i_k} \nn \\        &=\phi_{m}(y_{i_k};\hat y_{k-1}) + \frac{1}{2\lam} \|y_{i_k} - \hat y_{k-1}\|^2  - \Gamma_{i_k}(x_{i_k}) - \frac{1}{2\lambda}\|x_{i_k}-\hat y_{k-1}\|^2. \label{ineq:tik-1}
	    \end{align}
        Noting from the observation \eqref{eq:minseq} that
        \[
        \phi_{m}(y_j;\hat y_{k-1}) + \frac{1}{2\lam} \| y_j -\hat y_{k-1}\|^2 
   \le \phi_{ m}(x_j;\hat y_{k-1}) + \frac{1}{2\lam} \| x_j -\hat y_{k-1}\|^2.
        \]
        Hence, the above inequality with $j=i_k$ and \eqref{ineq:tik-1} imply that
        \[
        t_{i_k} \le \phi_{m}(x_{i_k};\hat y_{k-1})  - \Gamma_{i_k}(x_{i_k}) \stackrel{\eqref{eq:def-tildem},\eqref{eq:serious}} \le f_m(x_{i_k};\hat y_{k-1}) - \ell_{f}(x_{i_k};\hat y_{k-1}),
        \]
        where the second inequality follows from \eqref{eq:def-tildem}, \eqref{eq:serious} with $j+1$ replaced by $j=i_k$, and the definition of $\phi_m$  in \eqref{def:phim}.
    Moreover, applying \eqref{eq:fm-linear1} with  $z = \hat y_{k-1}$, we have
	\[
     f_{m}(x_{i_k};\hat y_{k-1}) - \ell_{ f}(x_{i_k};\hat y_{k-1}) \le 2  M \|x_{i_k}-\hat y_{k-1}\|  + \frac{ L + 
    m}2 \|x_{i_k}-\hat y_{k-1}\|^2.
	\]
    Combining the above two inequalities, we have
	    \begin{equation}\label{eq:im4}
	         t_{i_k} 
	        \le 2 M \|x_{i_k}-\hat y_{k-1}\| + \frac{ L+m }{2}\|x_{i_k}-\hat y_{k-1}\|^2 
	        \le  M^2 + \left( \frac{ L+m}{2}+1\right)\|x_{i_k}-\hat y_{k-1}\|^2,
	    \end{equation}
	  where the  last inequality is due to the fact that $2ab \le a^2+b^2$ with $a = M$ and $b = \|x_{i_k}-\hat y_{k-1}\|$.

	  We will now  bound $\|x_{i_k}-\hat y_{k-1}\|^2$.
 It follows from  \eqref{eq:wvar_bdd}
 that
	  \[
	      \frac{1}{2}\left(m + \frac{1}{\lam}\right)\|\hat y_{k}-\hat y_{k-1}\|^2 \le \hat M^{\lam}(\hat x_0) - \phi(\hat y_{k})     + 2\delta. 
	  \]
   which, in view of the definition of $\beta_1$ in \eqref{beta}, is equivalent to
	  \[
	      \phi(\hat y_k) - \phi^* + \frac{1}{2\beta_1}\left(m + \frac{2}{\zeta \lam}\right)\|\hat y_{k}-\hat y_{k-1}\|^2 \le \hat M^{\lam}(\hat x_0) - \phi^* + 2\delta.
	  \]
      It is easy to verify that $\beta_1\ge 1$ from its definition in \eqref{beta}.
      Clearly, this observation, the fact that $\phi(\hat y_k) \ge \phi^*$, and the above inequality imply that
        \begin{equation}
	      \phi(\hat y_k) - \phi^* + \frac{1}{2}\left(m + \frac{2}{\zeta \lam}\right)\|\hat y_{k}-\hat y_{k-1}\|^2 \le \beta_1[\hat M^{\lam}(\hat x_0) - \phi^* + 2\delta].
	      \label{eq:im2}
	  \end{equation}
      Using this inequality and
       Lemma \ref{lem:bdd_smdist} with $(\Gamma,z_0,u) = (\Gamma_{i_k},\hat y_{k-1},\hat y_k)$, we obtain
	  \begin{align*}
	   \zeta^2\left(\frac{1}{4\zeta \lam}+\frac{m}{2} \right)  \|x_{i_k}- \hat y_{k-1} \|^2
	&\stackrel{\eqref{eq:uniform_bdd}}\le
	\phi(\hat y_k) - \phi ^ * + \frac12 \left(m+\frac{2}{\zeta \lam} \right) \|\hat y_k - \hat y_{k-1}\|^2 + 4\zeta \lam M^2   \\
 &\stackrel{\eqref{eq:im2}} \le \beta_1 [ \hat M^{\lam}(\hat x_0)- \phi^* + 2\delta] + 4\zeta \lam M^2.
 % &\le 2\beta_1  [ \hat M^{\lam}(\hat x_0)- \phi^* ] \label{eq:im3}
	  \end{align*}
 
	The statement now follows by combining \eqref{eq:im4} and the above inequality, and using the definitions of $\beta_2$ and $\bar t$ in \eqref{beta} and \eqref{def:bar t}, respectively.
 
 b) This statement follows from a) and the definitions of $\sigma_{k}$ and $\bar \sigma$ in \eqref{def:tau} and \eqref{def:bar tau}, respectively.
 \end{proof}
 
% We are now ready to present the proof of Theorem \ref{thm:main1}.

We end this section by noting that
the proof of Theorem \ref{thm:main1}
 now follows immediately from Propositions~\ref{prop:innercomp} and \ref{thm:outer}, and Lemma \ref{lem:t1}(b).
% The conclusion of
% the theorem now follows from the choice of $\sigma$ in \eqref{eq:equaltau}. 

\section{Computational Results}\label{num}
This section reports computational results comparing the PBF method of
Subsection~\ref{subsec:update} with the PS method of
Subsection~\ref{SEC:SCS} and the Redistributed Proximal Bundle (RPB) method
of \cite{hare2010redistributed}. It contains three subsections. The first
presents computational results for the phase retrieval problem, the second
for the blind deconvolution problem, and the third for the constrained MAXQUAD problem.
% and solving Lagrangian cut as an interesting application integer programming.

% This section is organized into three subsections. The first subsection provides details on the
% implementation of Ad-GPB. The remaining two subsections describe the results of the computational experiments mentioned above.
% \subsection{Implementation details}
% We start by describing how we choose parameters and how we implement BUF for Ad-GPB. For Ad-GPB and Ad-GPB, First, we set $\tau = 0.95$ and $\beta_0 = 1/2$. We point out that the following two applications both have the feature that the optimal value is known. This fact and  Lemma \ref{iterate} suggest that we don't need to compute all the quantities in \eqref{def:phik}, \eqref{def:gk} and \eqref{def:nk}. This means we actually implement a simpler version because the test after \eqref{def:nk} is not performed. Second, we implement BUF by two-cut scheme (S2) and the details are stated in section \ref{Algorithm}.

% Add sparse problem 

Both phase retrieval and blind deconvolution  are special cases of \eqref{eq:opt_problem} where $h\equiv0$ and $f(\cdot)$ has the form
\begin{equation}\label{fg}
f(x)=g(c(x))
\end{equation}
for some convex and $L$-Lipschitz function $g$  and smooth map $c$ with $\beta$-Lipschitz Jacobian.
It follows from  Lemma 4.2 in \cite{drusvyatskiy2019efficiency} that such a function $f(\cdot)$  is $L\beta$-weakly convex and hence we set $m = L\beta$. Recall that the PS method terminates based on Moreau stationary point (see Definition \ref{def:Moreau}) which is hard to verify in view of the remark below Proposition \ref{prop:Davis}. On the other hand, it is still an open problem for the PS method to establish the convergence to find a  regularized stationary point (see Definition~\ref{def:appr}).
% Note that the PS method and PBF do not have a common theoretical termination condition: the PS method terminates based on the Moreau Envelope (see Definition \ref{def:Moreau}), while PBF terminates based on the regularized stationary point (see Definition~\ref{def:appr}).
Hence for the sake of our numerical experiments, we adopt the termination criterion 
\begin{equation}\label{rela-var}
\phi(x_k) - \phi_* \le \tilde \varepsilon [ \phi(x_0)-\phi_*]
\end{equation}
that requires the knowledge of $\phi_*$ where either $\tilde \varepsilon =  10^{-3}$ or $10^{-4}$. 

 Now we provide details of the algorithms used in the following three subsections.  
We implement the constant stepsize version of PS method, i.e.,  $\alpha_t = \alpha$ for every $t \ge 0$. In our experiment, we choose four values of $\alpha$, namely,  $\alpha = 1/(32m)$, $1/(8m)$, $1/(2m)$, and $1/m$. Hence in our experiment: (i) the choice of $\alpha$ for the PS method does not follow  ${\cal{ O}}(1/\sqrt{T})$ recipe of Proposition \ref{prop:Davis} where $T$ is the total number of PS iterations; (ii)   $T$ is determined by \eqref{rela-var} rather than being pre-specified.
PBF sets $\lam = 1/2m$ and
       % Motivated by the termination criterion \eqref{rela-var}, PBF uses a different choice for $\delta$ compared to the theoretical one  (see step 0 of PBF), given by
       $\delta = \tilde \varepsilon [ \phi(x_0)-\phi_*]$. 
    PBF is implemented based on the 2-cut scheme (see Subsection \ref{sec:BUF}). For phase retrieval  and  blind deconvolution problems, due to the fact that $h=0$, the subproblem \eqref{def:subprob} in the 2-cut scheme has a closed-form solution. We implement RPB according to \cite{hare2010redistributed} and set $M_0=10$, $R_0=10$, the Armijo
parameter to $0.05$, and the convexification growth parameter to $2$.
It is implemented using the three-cut scheme, which retains the newly
generated cut, the cut associated with the current serious point, and the
aggregate cut obtained from the preceding bundle. The resulting proximal
subproblem is formulated as a convex quadratic program and solved using
MOSEK 10.2\footnote{https://docs.mosek.com/latest/toolbox/index.html}.
% All experiments were performed in MATLAB 2023a and run on a PC with a 16-core Intel Core i9 processor and 32 GB of memor
% Now we describe the algorithm details used in the following two subsections. We first describe the subgradient method.
% Given $x_k$, it computes
% namely  Pol-Sub, which is a special case of subgradient method with Polyak stepsize as:
 % \[
 %     x_{k+1} = \argmin_x \left\{\ell_{\phi} (x;x_k)+\frac{1}{2\lam}\|x-x_k\|^2\right\}.
 %     \]
     % where $g(x):= A^T \operatorname{sign}(Ax-b) \in \partial f(x)$ and 
     % PBF is as in Subsection \ref{subsec:update} with 
     % \[
     % \delta = \tilde \varepsilon [ \phi(x_0)-\phi_*], \quad \lam = \frac{5}{m}
     % \]
     % where $m=$.
     
Finally, all experiments were performed in MATLAB 2023a and run on a PC with a 16-core Intel Core i9 processor and 32 GB of memory.

% We now describe some details about all the tables that appear in this subsection.  We set the target $\bar \varepsilon$ in \eqref{rela-var} as $10^{-5}$ for dense instances and  $10^{-4}$ for sparse instances. 
% The quantities $\theta_m$, $\theta_n$ and $\theta_s$  are defined as $\theta_{m}  = m/10^3$, $\theta_{n} =  n/10^3$, and $\theta_s := \text{nnz}(A)/mn$, where $\text{nnz}(A)$ is the number of non-zero entries of $A$.  
% %  The second to tenth columns provide numbers of
% % iterations and running times for the three methods with three initial prox stepsize choices. 
% An entry in each table is given as a fraction with the numerator expressing the (rounded) number of iterations and the  denominator expressing the CPU running time in seconds. An entry marked as $*/*$ indicates that the CPU running time exceeds the allocated time limit. The
% bold numbers highlight the method that has the best performance for each instance.
\subsection{Phase retrieval}\label{pr}
This subsection reports  computational results comparing PBF  against  PS and RPB on the phase retrieval problem.

Consider the  Phase Retrieval problem (see Subsection 5.1 of  \cite{davis2019stochastic})
\begin{equation}\label{phase}
\min _{x \in \mathbb{R}^d} \left\{f(x) := \frac{1}{n}\sum_{i=1}^nf_i(x)\right\}, \quad f_i(x):= \left|\left\langle a_i, x\right\rangle^2-b_i\right|
\end{equation}
% \[
% h = |\cdot|
% \]
% \[
% c_i(x) = x^\top (a_i a_i^\top) x - b_i
% \]
% \[
% \nabla c_i(x) = 2(a_i a_i^\top) x
% \]
% \[
% L_i = 2\|a_i a_i^\top\|
% \]
% \[
% m = \frac{1}{n} \sum L_i
% \]
% %\textbf{Agree}
% \begin{equation}
% \min _{x \in \mathbb{R}^d} g(x):=\frac{1}{n} \sum_{i=1}^n\left|\left\langle a_i, x\right\rangle^2-b_i\right| .
% \end{equation}
% \begin{equation}
% \partial g_i(x)=2\langle a
% _i, x\rangle a_i \cdot\left\{\begin{array}{lr}
% \operatorname{sign}\left(\langle a_i, x\rangle^2-b\right), & \text { if }\langle a, x\rangle^2 \neq b \\
% {[-1,1],} & \text { o.w. }
% \end{array}\right\} .
% \end{equation}
where $a_i\in \R^d$ and $b_i \in \R$. 
% Note that \eqref{phase} is a special instance of \eqref{eq:opt_problem} with $h\equiv0$. 
Observe that $f_i(\cdot)$ is a special case of \eqref{fg} with $g(\cdot)= |\cdot|$ and $c(\cdot) =\left\langle a_i, \cdot\right\rangle^2-b_i $. Thus $f_i$ is $2\|a_i a_i^\top\|$-weakly convex, and we set 
\[
m  =\frac{2}{n}\sum_{i=1}^n\|a_i a_i^\top\|.
\]
 Now we describe the data generation process. Vectors $\{a_i\}$ are i.i.d. generated according to standard Gaussian measurements $ N\left(0, I_{d \times d}\right)$. Then we generate the target signal $\bar{x}$ and initial point $x_0$ uniformly on the unit sphere; and set $b_i=\left\langle a_i, \bar{x}\right\rangle^2$ for each $i=1, \ldots, n$.  Thus the optimal value $\phi_*=0$ for \eqref{phase}.
We perform four sets of experiments corresponding to $(d, n)=(100,300),(200,600)$, $(500,1500)$,$(1000,3000)$ and record the results in Tables \ref{tabfirstpb:cpuobj1} and \ref{tabfirstpb:cpuobj2}.

 We now describe some details about  the tables that appear in this subsection.   The second to fifth columns provide numbers of
 iterations and running times for the six variants. 
 An entry in each table has two numbers where the first one expresses the number of iterations and the  second one expresses the CPU running time in seconds. An entry marked as $*/*$ indicates that the CPU running time exceeds the one hour time limit. 
% Two cut subproblem
% \[
% \Gamma_1(x) = \max\{l(x;x_0),l(x;x_1)\}
% \]

% \[
% l(x;x_0) =  f(x_0) + \inner{g(x_0)}{x- x_0} = \inner{a_1}{x}+b_1
% \]

% \[
% a_1=g(x_0), \quad b_1 = f(x_0) - \inner{g(x_0)}{x_0}
% \]

% \[
% l(x;x_1) =  f(x_1) + \frac{m}{2}\|x_1 -x_0\|^2 + \inner{g(x_1) + m(x_1-x_0)}{x- x_1} = \inner{a_2}{x}+b_2
% \]

% \[
% a_2 = g(x_1) + m(x_1-x_0), \quad b_2 = f(x_1)+ \frac{m}{2}\|x_1-x_0\|^2 - \inner{g(x_1)}{x_1} - m\inner{x_1 - x_0}{x_1}
% \]

% {\bf Task:} Compare our code based on the $1$-cut and $2$-cut 
% with the subgradient-type algorithms of papers [3] and [4]

% Question 1 
% 1000 iterations:
% subgradient 0.5/m 1805  1/m 930
% PBF Delta:5 PERCENT (1000) 944

% Question 2
% 10000 iterations:
% subgradient 0.5/m 48 1/m 231
% PBF Delta:0.5 PERCENT (100) 42
\begin{table}[H]
\centering
\begin{tabular}{|c|c|c|c|c|}
\hline
\multicolumn{1}{|c|}{-} &
\multicolumn{1}{|c|}{$L_1: (100,300)$}&
\multicolumn{1}{|c|}{$L_2: (200,600)$}&
\multicolumn{1}{|c|}{ $L_3: (500,1500)$}&
\multicolumn{1}{|c|}{ $L_4: (1000,3000)$}\\
\hline
\multirow[t]{10}{*}{PS: $\alpha = 1/(32m)$} 
%5  &14.6169   &0.001 & 14.6126 & 0.107 & 14.6465 & 0.0007 & 14.6904 &0.04 \\
%&10&    14.6155 & 0.0012 &14.6113  & 0.109 &14.6449  & 0.001 &14.6892& 0.05\\
&50852/1.48 & 203416/17.82 &*/*& */*\\
\hline
\multirow[t]{10}{*}{PS: $\alpha = 1/(8m)$} 
%5  &14.6169   &0.001 & 14.6126 & 0.107 & 14.6465 & 0.0007 & 14.6904 &0.04 \\
%&10&    14.6155 & 0.0012 &14.6113  & 0.109 &14.6449  & 0.001 &14.6892& 0.05\\
&14852/0.47 &  64294/5.58 &432176/201.12& 235700/650.21\\
\hline
\multirow[t]{10}{*}{PS: $\alpha = 1/(2m)$} 
%5  &14.6169   &0.001 & 14.6126 & 0.107 & 14.6465 & 0.0007 & 14.6904 &0.04 \\
%&10&    14.6155 & 0.0012 &14.6113  & 0.109 &14.6449  & 0.001 &14.6892& 0.05\\
& 4436/0.12 & 20011/1.76 &345346/167.21& 79165/216.67\\
\hline
\multirow[t]{10}{*}{PS: $\alpha = 1/m$} 
%5  &14.6169   &0.001 & 14.6126 & 0.107 & 14.6465 & 0.0007 & 14.6904 &0.04 \\
%&10&    14.6155 & 0.0012 &14.6113  & 0.109 &14.6449  & 0.001 &14.6892& 0.05\\
&*/* & */* &*/*& 67095/183.56\\
\hline
\multirow[t]{10}{*}{PBF} 
%5 &12.8535 & 0.001 & 13.3252 & 0.0012 & 13.9843 &0.001  & 13.7750 &0.004\\
%&10&12.3639 & 0.0019 & 12.8431& 0.002& 13.9539 & 0.003 & 13.7763 &0.008\\
&5088/0.16 &24486/2.08&316782/154.21& 87514/242.62\\
\hline

\multirow[t]{10}{*}{RPB} 
%5&13.1016&0.001  & 12.7594 & 0.002 & 13.7884 &0.001  &  13.8973&0.009\\
%&10&12.5790&0.0015  & 11.2559 & 0.003 & 13.5968 & 0.002 & 13.7777 &0.01\\
&9318/17.08 & */* & */*& */*\\
\hline
\end{tabular}
\caption{PS versus PBF and RPB on \eqref{phase}. A relative tolerance of $\tilde \varepsilon  = 10^{-3}$ is set.}
\label{tabfirstpb:cpuobj1}
\end{table}

\begin{table}[H]
\centering
\begin{tabular}{|c|c|c|c|c|}
\hline
\multicolumn{1}{|c|}{-} &
\multicolumn{1}{|c|}{$L_1: (100,300)$}&
\multicolumn{1}{|c|}{$L_2: (200,600)$}&
\multicolumn{1}{|c|}{ $L_3: (500,1500)$}&
\multicolumn{1}{|c|}{ $L_4: (1000,3000)$}\\
\hline
\multirow[t]{10}{*}{PS: $\alpha = 1/(32m)$} 
%5  &14.6169   &0.001 & 14.6126 & 0.107 & 14.6465 & 0.0007 & 14.6904 &0.04 \\
%&10&    14.6155 & 0.0012 &14.6113  & 0.109 &14.6449  & 0.001 &14.6892& 0.05\\
&72376/2.11 & 380165/31.32 &653256/305.65& */*\\
\hline
\multirow[t]{10}{*}{PS: $\alpha = 1/(8m)$} 
%5  &14.6169   &0.001 & 14.6126 & 0.107 & 14.6465 & 0.0007 & 14.6904 &0.04 \\
%&10&    14.6155 & 0.0012 &14.6113  & 0.109 &14.6449  & 0.001 &14.6892& 0.05\\
&20328/0.61 & 118589/9.91 & */*
& 424433/1.03E+03\\
\hline
\multirow[t]{10}{*}{PS: $\alpha = 1/(2m)$} 
%5  &14.6169   &0.001 & 14.6126 & 0.107 & 14.6465 & 0.0007 & 14.6904 &0.04 \\
%&10&    14.6155 & 0.0012 &14.6113  & 0.109 &14.6449  & 0.001 &14.6892& 0.05\\
&*/* & */* &*/*& */*\\
\hline
\multirow[t]{10}{*}{PS: $\alpha = 1/m$} 
%5  &14.6169   &0.001 & 14.6126 & 0.107 & 14.6465 & 0.0007 & 14.6904 &0.04 \\
%&10&    14.6155 & 0.0012 &14.6113  & 0.109 &14.6449  & 0.001 &14.6892& 0.05\\
&*/* & */* &*/*& */*\\
\hline
\multirow[t]{10}{*}{PBF} 
%5 &12.8535 & 0.001 & 13.3252 & 0.0012 & 13.9843 &0.001  & 13.7750 &0.004\\
%&10&12.3639 & 0.0019 & 12.8431& 0.002& 13.9539 & 0.003 & 13.7763 &0.008\\
&15228/0.42&125603/10.61&1056034/454.93& 742553/1.86E+03
\\
\hline
\multirow[t]{10}{*}{RPB} 
%5&13.1016&0.001  & 12.7594 & 0.002 & 13.7884 &0.001  &  13.8973&0.009\\
%&10&12.5790&0.0015  & 11.2559 & 0.003 & 13.5968 & 0.002 & 13.7777 &0.01\\
&*/* & */*&*/*& */*\\
\hline
\end{tabular}
\caption{PS versus  PBF and RPB on \eqref{phase}. A relative tolerance of $\tilde \varepsilon  = 10^{-4}$ is set.}
\label{tabfirstpb:cpuobj2}
\end{table}
The two tables show that PS is sensitive to the choice of stepsize, with the fastest stepsize varying across tolerance levels and problem dimensions. In contrast, PBF uses a single fixed stepsize and achieves performance comparable to the best-performing PS stepsize among the four tested choices. RPB reaches the prescribed tolerance only for the smallest instance \(\tilde{\varepsilon}=10^{-3}\) and reaches the time limit in all other instances.

\subsection{Blind deconvolution}\label{BD}
This subsection reports  computational results comparing PBF  against PS and RPB on the blind deconvolution problem.
Consider the blind deconvolution problem (see Subsection 5.2 of  \cite{davis2019stochastic}) 
\begin{equation}\label{blind}
\min _{x, y} \left\{f(x) :=\frac{1}{n} \sum_{i=1}^n f_i (x)\right\}, \quad f_i(x):=\left|\left\langle u_i, x\right\rangle\left\langle v_i, y\right\rangle-b_i\right|
\end{equation}
where $u_i, v_i \in \R^d$ and $b_i \in \R$. Note that \eqref{blind} is a special instance of \eqref{eq:opt_problem} with $h\equiv0$. Observe that $f_i(\cdot)$ is a special case of \eqref{fg} with $g(\cdot)= |\cdot|$ and $c(x,y) =\left\langle u_i, x\right\rangle\left\langle v_i, y\right\rangle-b_i $. Thus $f_i(\cdot)$ is $\|u_i v_i^\top\|$-weakly convex, and we set 
\[
m  =\frac{1}{n}\sum_{i=1}^n\|u_i v_i^\top\|.
\]
Now we describe the data generation process. Vectors $\{u_i\}$,  $\{v_i\}$ are i.i.d. generated Gaussian measurements $ N\left(0, I_{d \times d}\right)$. Then we generate the target signals $\bar{x},\bar y$ and initial points $x_0,y_0$ uniformly on the unit sphere; and set $b_i=\left\langle u_i, \bar x\right\rangle\left\langle v_i, \bar y\right\rangle$ for each $i=1, \ldots, n$. 
Thus the optimal value $\phi_*=0$ for \eqref{blind}.
We perform four sets of experiments corresponding to $(d, n)=(100,300),(200,600)$, $(500,1500)$,$(1000,3000)$ and record the results in Tables \ref{tabfirstpb:cpuobj3} and \ref{tabfirstpb:cpuobj4}. The explanations of the tables are the same as in Subsection \ref{pr}.
\begin{table}[H]
\centering
\begin{tabular}{|c|c|c|c|c|}
\hline
\multicolumn{1}{|c|}{-} &
\multicolumn{1}{|c|}{$L_1: (100,300)$}&
\multicolumn{1}{|c|}{$L_2: (200,600)$}&
\multicolumn{1}{|c|}{ $L_3: (500,1500)$}&
\multicolumn{1}{|c|}{ $L_4: (1000,3000)$}\\
\hline
\multirow[t]{10}{*}{PS: $\alpha = 1/(32m)$} 
%5  &14.6169   &0.001 & 14.6126 & 0.107 & 14.6465 & 0.0007 & 14.6904 &0.04 \\
%&10&    14.6155 & 0.0012 &14.6113  & 0.109 &14.6449  & 0.001 &14.6892& 0.05\\
&76282/2.21 & 812578/143.48 &*/*& */*\\
\hline
\multirow[t]{10}{*}{PS: $\alpha = 1/(8m)$} 
%5  &14.6169   &0.001 & 14.6126 & 0.107 & 14.6465 & 0.0007 & 14.6904 &0.04 \\
%&10&    14.6155 & 0.0012 &14.6113  & 0.109 &14.6449  & 0.001 &14.6892& 0.05\\
&16852/4.41 &  224893/39.62 &121091/350.21& 343700/740.55\\
\hline
\multirow[t]{10}{*}{PS: $\alpha = 1/(2m)$} 
%5  &14.6169   &0.001 & 14.6126 & 0.107 & 14.6465 & 0.0007 & 14.6904 &0.04 \\
%&10&    14.6155 & 0.0012 &14.6113  & 0.109 &14.6449  & 0.001 &14.6892& 0.05\\
& 5425/0.42 & */* &92137/241.41& 83125/288.21\\
\hline
\multirow[t]{10}{*}{PS: $\alpha = 1/m$} 
%5  &14.6169   &0.001 & 14.6126 & 0.107 & 14.6465 & 0.0007 & 14.6904 &0.04 \\
%&10&    14.6155 & 0.0012 &14.6113  & 0.109 &14.6449  & 0.001 &14.6892& 0.05\\
&*/* & */* &*/*& 64035/172.12\\
\hline
\multirow[t]{10}{*}{PBF} 
%5 &12.8535 & 0.001 & 13.3252 & 0.0012 & 13.9843 &0.001  & 13.7750 &0.004\\
%&10&12.3639 & 0.0019 & 12.8431& 0.002& 13.9539 & 0.003 & 13.7763 &0.008\\
&5021/0.39 &344221/72.68&101299/280.67& 77511/212.78\\
\hline
\multirow[t]{10}{*}{RPB} 
%5&13.1016&0.001  & 12.7594 & 0.002 & 13.7884 &0.001  &  13.8973&0.009\\
%&10&12.5790&0.0015  & 11.2559 & 0.003 & 13.5968 & 0.002 & 13.7777 &0.01\\
&*/* & */*&*/*& */*\\
\hline
\end{tabular}
\caption{PS versus  PBF and RPB on \eqref{blind}. A relative tolerance of $\tilde \varepsilon  = 10^{-3}$ is set.}
\label{tabfirstpb:cpuobj3}
\end{table}
\begin{table}[H]
\centering
\begin{tabular}{|c|c|c|c|c|}
\hline
\multicolumn{1}{|c|}{-} &
\multicolumn{1}{|c|}{$L_1: (100,300)$}&
\multicolumn{1}{|c|}{$L_2: (200,600)$}&
\multicolumn{1}{|c|}{ $L_3: (500,1500)$}&
\multicolumn{1}{|c|}{ $L_4: (1000,3000)$}\\
\hline
\multirow[t]{10}{*}{PS: $\alpha = 1/(32m)$} 
%5  &14.6169   &0.001 & 14.6126 & 0.107 & 14.6465 & 0.0007 & 14.6904 &0.04 \\
%&10&    14.6155 & 0.0012 &14.6113  & 0.109 &14.6449  & 0.001 &14.6892& 0.05\\
&152976/6.12 & 1701581/300.33 &*/*& */*\\
\hline
\multirow[t]{10}{*}{PS: $\alpha = 1/(8m)$} 
%5  &14.6169   &0.001 & 14.6126 & 0.107 & 14.6465 & 0.0007 & 14.6904 &0.04 \\
%&10&    14.6155 & 0.0012 &14.6113  & 0.109 &14.6449  & 0.001 &14.6892& 0.05\\
&41328/1.39 & 283353/50.32 & 173291/670.81
& 824121/1.93E+03\\
\hline
\multirow[t]{10}{*}{PS: $\alpha = 1/(2m)$} 
%5  &14.6169   &0.001 & 14.6126 & 0.107 & 14.6465 & 0.0007 & 14.6904 &0.04 \\
%&10&    14.6155 & 0.0012 &14.6113  & 0.109 &14.6449  & 0.001 &14.6892& 0.05\\
&*/* & */* &*/*& */*\\
\hline
\multirow[t]{10}{*}{PS: $\alpha = 1/m$} 
%5  &14.6169   &0.001 & 14.6126 & 0.107 & 14.6465 & 0.0007 & 14.6904 &0.04 \\
%&10&    14.6155 & 0.0012 &14.6113  & 0.109 &14.6449  & 0.001 &14.6892& 0.05\\
&*/* & */* &*/*& */*\\
\hline
\multirow[t]{10}{*}{PBF} 
%5 &12.8535 & 0.001 & 13.3252 & 0.0012 & 13.9843 &0.001  & 13.7750 &0.004\\
%&10&12.3639 & 0.0019 & 12.8431& 0.002& 13.9539 & 0.003 & 13.7763 &0.008\\
&35227/1.21 &394212/89.32&188823/690.21& 1142572/2.65E+03
\\
\hline
\multirow[t]{10}{*}{RPB} 
%5&13.1016&0.001  & 12.7594 & 0.002 & 13.7884 &0.001  &  13.8973&0.009\\
%&10&12.5790&0.0015  & 11.2559 & 0.003 & 13.5968 & 0.002 & 13.7777 &0.01\\
&*/* & */* &*/*& */*\\
\hline
\end{tabular}
\caption{PS versus  PBF and RPB on \eqref{blind}. A relative tolerance of $\tilde \varepsilon  = 10^{-4}$ is set.}
\label{tabfirstpb:cpuobj4}
\end{table}

The two tables show that PS is sensitive to the choice of stepsize, with the best-performing stepsize in terms of CPU time varying across tolerance levels and problem dimensions. In contrast, PBF uses a single fixed stepsize and performs comparably to the best of the four tested PS stepsizes. RPB does not reach the
prescribed tolerance within the time limit.

\subsection{MAXQUAD}

This subsection reports computational results comparing PBF, PS, and RPB method on the constrained MAXQUAD problem
\begin{equation}\label{eq:maxquad}
    \min_{x\in\mathbb{R}^d}\{\phi(x):=f(x)+h(x)\},
    \qquad
    f(x):=\max_{1\le j\le p}\{q_j(x):=x^\top A_jx+b_j^\top x\},
\end{equation}
where $h$ is the indicator function of the ball $\{x:\|x\|\le D\}$. We set
$p=5$ and $D=2$.

We now describe the data generation. Let $u_1,\ldots,u_p$ be the vertices of
a regular simplex in $\mathbb{R}^d$, satisfying
$\sum_{j=1}^p u_j=0$, $\|u_j\|=1$, and
$\langle u_i,u_j\rangle=-1/(p-1)$ for $i\neq j$. Let $
P:=(p-1)\sum_{j=1}^p u_j u_j^\top/p$,
which is the orthogonal projector onto
$\operatorname{span}\{u_1,\ldots,u_p\}$, and let $x^*$ be a randomly generated
unit vector orthogonal to this subspace. For $\beta>0$, we set
\[
    \kappa=\frac{2\beta(p-1)}{p-2},\qquad
    A_j=\beta I+\kappa\left(\frac{P}{p-1}-u_ju_j^\top\right),\qquad
    b_j=u_j-2\beta x^*.
\]

Since $\lambda_{\min}(A_j)=-\beta$, every $q_j$, and hence $f$, is
$2\beta$-weakly convex. We set $\beta=1$, and hence $m=2$. Moreover,
$\sum_j A_j/p=\beta I$ and $\sum_j b_j/p=-2\beta x^*$, which imply
\[
    f(x)\ge \frac1p\sum_{j=1}^p q_j(x)
    =\beta\|x-x^*\|^2-\beta.
\]
Since all $q_j(x^*)=-\beta\|x^*\|^2 = -\beta$ and $\|x^*\|=1<D$, we have
$\phi_*=\phi(x^*)=-\beta=-1$.

The initial point $x_0$ is generated from a standard Gaussian vector, scaled
by $1/\sqrt d$, and projected onto the feasible ball if necessary. We test
$d\in\{20,50,100,200\}$. Each method is given an independent one-hour time limit. Each table entry
reports the number of iterations and CPU time in seconds; $\ast/\ast$
indicates that the time limit was reached before the target accuracy was
attained.
Tables~\ref{tab:maxquad-1e-3} and~\ref{tab:maxquad-1e-4} report the results for
$\tilde{\varepsilon}=10^{-3}$ and $10^{-4}$, respectively.

\begin{table}[H]
\centering
\begin{tabular}{|c|c|c|c|c|}
\hline
\multicolumn{1}{|c|}{-}
& \multicolumn{1}{c|}{$L_1:(20)$}
& \multicolumn{1}{c|}{$L_2:(50)$}
& \multicolumn{1}{c|}{$L_3:(100)$}
& \multicolumn{1}{c|}{$L_4:(200)$} \\
\hline
PS: $\alpha=1/(32m)$
& $16636/0.29$ & $17183/0.57$ & $17260/1.27$ & $17257/11.25$ \\
\hline
PS: $\alpha=1/(8m)$
& $3385/0.06$ & $3482/0.11$ & $3502/0.26$ & $3507/3.02$ \\
\hline
PS: $\alpha=1/(2m)$
& $1732/0.03$ & $1775/0.06$ & $1784/0.13$ & $1793/1.41$ \\
\hline
PS: $\alpha=1/m$
& $\ast/\ast$ & $\ast/\ast$ & $\ast/\ast$ & $\ast/\ast$ \\
\hline
PBF
& $818/0.09$ & $750/0.07$ & $776/0.10$ & $855/0.78$ \\
\hline
RPB
& $\ast/\ast$ & $\ast/\ast$ & $\ast/\ast$ & $\ast/\ast$ \\
\hline
\end{tabular}
\caption{PS versus PBF and RPB on \eqref{eq:maxquad}. A relative tolerance
of $\tilde{\varepsilon}=10^{-3}$ is set.}
\label{tab:maxquad-1e-3}
\end{table}

\begin{table}[H]
\centering
\begin{tabular}{|c|c|c|c|c|}
\hline
\multicolumn{1}{|c|}{-}
& \multicolumn{1}{c|}{$L_1:(20)$}
& \multicolumn{1}{c|}{$L_2:(50)$}
& \multicolumn{1}{c|}{$L_3:(100)$}
& \multicolumn{1}{c|}{$L_4:(200)$} \\
\hline
PS: $\alpha=1/(32m)$
& $23101/0.37$ & $23518/0.75$ & $23620/1.72$ & $23697/20.57$ \\
\hline
PS: $\alpha=1/(8m)$
& $\ast/\ast$ & $\ast/\ast$ & $\ast/\ast$ & $\ast/\ast$ \\
\hline
PS: $\alpha=1/(2m)$
& $\ast/\ast$ & $\ast/\ast$ & $\ast/\ast$ & $\ast/\ast$ \\
\hline
PS: $\alpha=1/m$
& $\ast/\ast$ & $\ast/\ast$ & $\ast/\ast$ & $\ast/\ast$ \\
\hline
PBF
& $8435/0.41$ & $7491/0.51$ & $7702/0.95$ & $8571/8.76$ \\
\hline
RPB
& $\ast/\ast$ & $\ast/\ast$ & $\ast/\ast$ & $\ast/\ast$ \\
\hline
\end{tabular}
\caption{PS versus PBF and RPB on \eqref{eq:maxquad}. A relative tolerance
of $\tilde{\varepsilon}=10^{-4}$ is set.}
\label{tab:maxquad-1e-4}
\end{table}

The two tables show that PBF requires fewer iterations than PS across the
tested dimensions. Moreover, PS is sensitive to the choice of stepsize,
whereas PBF uses the same stepsize in all instances. RPB does not reach the
prescribed tolerance within the time limit.

\section{Concluding Remarks}
\label{sec:conclusion}

% This paper presents a proximal bundle framework for solving HWC-CO problem and establishes the first iteration-complexity for PB methods for solving ?????weakly convex problems. Instead of focusing on a specific bundle update scheme, PBF introduces BUF which contains various well-known update schemes. PBF generates a sequence of iterates $\{\hat y_k\}$ with the corresponding residuals $(\hat w_k,\hat \varepsilon_k)$ such that $(x,w,\varepsilon)=(\hat y_k,\hat w_k,\hat \varepsilon_k)$
%  satisfies the inclusion in \eqref{eq:app_sol}.
% The iteration complexity is then established based on measuring the norm of $\hat w_k$ and $\hat \varepsilon_k$ simultaneously. ??? 

In this section, we provide some further remarks and directions for future research.

First, from the point of view of the sequences of
serious iterates $\{\hat x_k\}$ and $\{\hat y_k\}$,
the complexity result for
PBF is point-wise since it
is about a single iterate
from $\{\hat y_k\}$.
It would be also interesting to
establish an ergodic complexity result about a weighted average of
such sequence. 

Second, as already observed in the fifth remark following PBF, PBF
requires the knowledge of a
weakly convex parameter,
i.e., a scalar $m$ as in
Assumption (A1), and hence is not a universal method.
It would be interesting to develop an adaptive method
which do not require
a scalar $m$ as above, but instead generates adaptive estimates for it which possibly violate the condition on (A1).

Third, paper
\cite{davis2019proximally} presents an
inexact proximal point method (see Algorithm 2 of  \cite{davis2019proximally}) for solving \eqref{eq:opt_problem} where the subgradient method (SM) is used to obtain an approximate solution $\hat x_k$ of \eqref{sub_pro}. More specifically, under the condition that
$j_k \ge 11 (1+\lam m)^2$, it is shown that
$j_k$ iterations of the SM initialized from $\hat y_{k-1}$ yield a point $\hat x_k$ satisfying
 \[
\phi(\hat x_k) - \phi(\hat y_{k-1}) \le \frac{72 \lam M^2}{j_k+1} - \frac1{4\lam} \|\hat y_{k-1}- x_k^*\|^2
\]
where $x_k^*$ is the minimizer of \eqref{sub_pro}.
Hence, if
$j_k$ is chosen as
\[
j_k = \bar j :=\left\lceil \max \left\{\frac{576 \lam^2 M^2}{ \varepsilon^2}, \frac{576 (1+\lam m)^2 M^2}{\varepsilon^2}, 11 (1+\lam m)^2 \right\}\right\rceil,
\]
then using \eqref{eq:new_grad}, we conclude that
\[
\frac{\lam^2}{(1+\lam m)^2}\|\nabla \hat M^{\lam}(\hat y_{k-1})\|^2 = \|\hat y_{k-1}-x_k^*\|^2 \le
     4\lam [\phi(\hat y_{k-1}) - \phi(\hat x_k)  ] + \frac{\lam^2 \varepsilon^2}{2(1+\lam m)^2}, 
\]
and hence that
\begin{equation}\label{ineq:Moreau-bound}
    \inf_{k \le K} \|\nabla \hat M^{\lam}(\hat y_{k-1})\|^2 \le \frac{4(1+\lam m)^2}{\lam K} [ \phi(x_0) - \phi(\hat x_K)] + \frac{\varepsilon^2}{2}.
\end{equation}
The above conclusion gives an easily computable bound on
the gradient of the Moreau envelope at $\hat y_{k-1}$.
A drawback of the above bound  is that it only holds if SM performs a prespecified number of iterations
$j_k \ge \bar j$.
% applied to \eqref{sub_pro}.
In contrast, it is shown 
in Appendix \ref{sec:termination}
that the iteration sequence  $\{(\hat x_k,\hat y_k)\}$ satisfies a
 bound similar to \eqref{ineq:Moreau-bound} even though it solves \eqref{sub_pro} only according to
 the stopping criterion \eqref{ineq:hpe1},
 and hence without performing a prespecified number of inner  iterations
 as SM does.\\
 
 % can be derived under the assumption that \eqref{??} is solved according to \eqref{??}. 
% \textbf{Funding}
% This work was partially supported by AFORS Grant FA9550-22-1-0088.

\bibliographystyle{plain}
\bibliography{Bundle_weakly}

\appendix

\section{Technical Results about Subdifferentials}
\label{APP:sub}

% \begin{lemma}
% \label{def: Frechet Subdifferential}
%  Let $\phi: \mathbb{R}^{n} \rightarrow \mathbb{R} \cup\{\infty\}$ be a proper closed function. At any point $x \in \operatorname{dom} \phi$, we have
% \[
% \partial \phi(x)=\left\{v \in \mathbb{R}^{n} \mid \phi(y) \geq \phi(x)+\langle v, y-x\rangle+o(\|y-x\|), \quad \forall y \in \R^n\right\}.
% \]
% \end{lemma}
% \begin{proof}
% On one hand, if $v \in \partial \phi(x)$ \red{???}, then the fact that $\forall y \in \R^n$
% \[
%   \phi(y) \ge \phi(x) + \inner{v}{y-x} + \min\{0,\phi(y)-\phi(x)-\inner{v} { y-x}\}.
% \]
% implies that $\liminf_{y \rightarrow x} \frac{\min\{0,\phi(y)-\phi(x)-\inner{v} { y-x}\}}{\|y-x\|} \geq 0$. Thus we know  \red{???}
% \[
% \min\{0,\phi(y)-\phi(x)-\inner{v} { y-x}\} = o(\|y-x\|).
% \]
% We then conclude
% \[
% \partial \phi(x) \subset \left\{v \in \mathbb{R}^{n} \mid\left( \forall y \in \mathbb{R}^{n}\right) \phi(y) \geq \phi(x)+\langle v, y-x\rangle+o(\|y-x\|)\right\}.
% \]
% On the other hand, if $v \in \left\{v \in \mathbb{R}^{n} \mid\left( \forall y \in \mathbb{R}^{n}\right) \phi(y) \geq \phi(x)+\langle v, y-x\rangle+o(\|y-x\|)\right\}$, then 
% \[
% \liminf _{y \rightarrow x} \frac{\phi(y)-\phi(x)-\left\langle v, y-x\right\rangle}{\|y-x\|} \geq \liminf _{y \rightarrow x}\frac{o(\|y-x\|)}{\|y-x\|} = 0.
% \]
% Thus the statement follows.
% ({\bf Jiaming and Honghao})
% % In particular, let $\varepsilon = 0$ in (7) and (8) we have the second statement. So we can rewrite (7) as
% % % \[
% % % \partial_\varepsilon \phi(x) = \partial \phi(x) + \varepsilon \bar B.
% % % \]
% % % %\label{prop:Fre_sub}
% \end{proof}

This section presents two
technical results about
$\varepsilon$-subdifferentials
that will be used in our analysis.

The first result describes
a simple relationship 
involving $\varepsilon$-subdifferentials
of $\phi_m(\cdot;x)$
for different points~$x$.
% The second result describes
% a simple relationship 
% involving $\varepsilon$-subdifferentials
% of $\phi_m(\cdot;x)$
% for different points $x$.
\begin{lemma}
If $\phi: \mathbb{R}^{n} \rightarrow \mathbb{R} \cup\{\infty\}$ is an $m$-weakly convex function, then
$\phi_m(\cdot;c)$ is convex for
every $c \in \R^n$.
Moreover,
for every $x \in \dom \phi$,
$c \in \R^n$, and
$\varepsilon \ge 0$, we have
\[
\partial_{\varepsilon} \left[\phi_m(\cdot;x) \right](x) = \partial_{\varepsilon} \left[\phi_m(\cdot;c) \right](x) -  m(x-c).
\]
\label{lem:chara_weakly}
\end{lemma}

\begin{proof}
% \[
% \phi_m(u;c)-\phi_m(x;c)
% =\phi_m(u) - \phi(x) + \frac{m}2 \|u-c\|^2 - \frac{m}2 \|x-c\|^2
% - \frac{m}2 \|u-x\|^2 
% = 
% =???? \phi_m(u;x)-\phi(x) - \inner{v}{u-x}
% \]
Let $x\in \dom \phi$, $c \in \R^n$, and
$\varepsilon \ge 0$ be given. Then,
using
the definition of $\phi_m(\cdot;\cdot)$
in \eqref{def:phim}, we easily see that
\[
\phi_m(u;c)-\phi_m(x;c) - \inner{v+m(x-c)}{u-x}
= \phi_m(u;x) - \phi(x)  - \inner{v}{u-x} \quad \forall  u, v \in \R^n.
\]
The result now follows from the above identity and the definition of the $\varepsilon$-subdifferential
in \eqref{def:subdif}.
\end{proof}

The second result describes the relationship between an $\varepsilon$-solution and its global minimizer for a strongly convex function.
\begin{lemma}
If $g$ is a proper lower semi-continuous $\mu$-strongly convex function and $y$ is an
$\varepsilon$-solution of $g$, i.e., $0 \in \partial_\varepsilon g(y) $, then
its global minimizer $\hat y$ satisfies
\begin{equation}\label{eq:B1}
0 \in \partial g(\hat y), \quad \|y-\hat y \| \le \sqrt{\frac{2\varepsilon}\mu}.
\end{equation}
\label{lem:mu convex}
\end{lemma}

\begin{proof}
The inclusion in \eqref{eq:B1} follows from the fact that
$\hat y$ is a global minimizer of $g$.
Since g is $\mu$-strongly convex and $\hat y$ is its global minimizer, we have 
\[
\frac{\mu}{2}\|y - \hat y\|^2 \le g(y) - g(\hat y) 
 \le \varepsilon
\]
where the second inequality is due to $0 \in \partial_\varepsilon g(y)$.
hence, the inequality in \eqref{eq:B1} follows.
\end{proof}

% \begin{lemma}
% Given two convex  functions $g$ and $h$ and a point $x\in \R^n$ such that $g \le h$ and $g(x) = h(x)$, then it holds that 
% \[
% \partial g(x) \subset \partial h(x).
% \]
% \end{lemma}

% \begin{lemma}
% if $g:\R^n \to \R^n$ is convex differentiable with $L$-Lipschitz continuous gradient and $v \in \partial_\varepsilon g(x)$ then
% \[
% \|v-\nabla g(x) \| \le 2 \sqrt{L\varepsilon}
% \]
% \end{lemma}

\section{Relationships Between Notions of Stationary Points}\label{App:relation}
This section contains three proofs for the results in Subsection~\ref{subsec:notion}.

\noindent
	{\bf Proof of Proposition \ref{prop:relation}}
 For this proof only, 
we  let
\begin{equation}\label{def:psi}
    \psi(\cdot) := \phi(\cdot) + \frac{\lam^{-1}+m}{2} \|\cdot-x\|^2
\end{equation}
and denote $\hat x^\lam(x)$ as in \eqref{eq:hatx}. Observe that $\psi$ is a $(1/\lam)$-strongly convex function in view of the fact that $\phi$ is $m$-weakly convex. Thus, for any $u \in \dom \psi$, it holds that
\begin{equation}\label{eq:psimu}
  \psi(u) - \psi(\hat x) \ge \frac{1}{2\lam}\|u-\hat x\|^2.  
\end{equation}
We start with the proof of a).

   a) 
Since $x$ is a $(\varepsilon_D,\delta_D)$-directional stationary point, there exists a $\tilde x$ such that
\begin{equation}\label{eq:DtoM}
\norm{x-\tilde x} \le \delta_D \quad 
\inf_{\|d\|\le 1} \phi'(\tilde x;d) \ge -\varepsilon_D.
\end{equation}
Then for any $d \in \R^n$,
we have
\begin{align*}
   \psi'(\tilde x;d) =
\phi'(\tilde x;d) + \left(\frac{1}{\lam}+m\right)\inner{\tilde x-x}{d} \ge -\varepsilon_D \|d\| - \left(\frac{1}{\lam}+m\right)\delta_D \|d\|
= - \left[\varepsilon_D + \left(\frac{1}{\lam}+m\right)\delta_D \right] \|d\|.
\end{align*}
Using the convexity of $\psi$  and the above relation with $d = \hat x^\lam(x)- \tilde x $, we conclude  that
\begin{equation}\label{eq:sped}
\psi(\hat x) - \psi(\tilde x) \ge 
\psi'(\tilde x;\hat x-\tilde x)
\ge 
- \left[\varepsilon_D + \left(\frac{1}{\lam}+m\right)\delta_D \right] \|\hat x-\tilde x\|.
\end{equation}
Then using the above inequality and \eqref{eq:psimu} with $u = \tilde x$ we can conclude that
\[
 \varepsilon_D + \left(\frac{1}{\lam}+m\right) \delta_D  \ge \frac{1}{2\lam} \|\tilde x-\hat x\|.
\]
Thus \eqref{eq:DtoM} further implies that 
\[
\|x-\hat x\| \le \norm{x-\tilde x} + \norm{\tilde x-\hat x} 
\le \delta_D + 2\lam \left[\varepsilon_D+\left(\frac{1}{\lam}+m\right)\delta_D \right].
\] 
and hence using \eqref{eq:new_grad}  we can get
\[
\|\nabla \hat M^{\lam}(x)\| = \left(m+\frac{1}{\lam}\right)\|\hat x^\lam(x)- x\| \le \left(m+\frac{1}{\lam}\right)\left[(3+2\lam m)\delta_D + 2\lam \varepsilon_D \right].
\]
Now the statement follows from the above inequality and the definition of Moreau stationary point in Definition \ref{def:Moreau}.

b) Since $x$ is a $(\varepsilon_M;\lam)$-Moreau stationary point, \eqref{eq:new_grad}  thus imply that 
\[
\left(\frac{1}{\lam}+m\right)\|x- \hat x\| \le \varepsilon_M.
\]
Then for any $d \in \R^n$ such that $\|d\| \le 1$, we have
\begin{align*}
   0 \le \psi'(\hat x;d) =
\phi'(\hat x;d) + \left(\frac{1}{\lam}+m\right)\inner{\hat x-x}{d} \le 
\phi'(\hat x;d)
+ \left(\frac{1}{\lam}+m\right) \norm{x-\hat x} \le \phi'(\hat x;d)+ \varepsilon_M
\end{align*}
and
\[
 \|x-\hat x\| \le \frac{\varepsilon_M}{m+1/\lam}.
\]
Choose $\tilde x = \hat x$, $\varepsilon_D = \varepsilon_M$ and 
$\delta_D =\varepsilon_M/(m+1/\lam)$.
Then the statement follows from the above relation and the definition of $(\varepsilon_D,\delta_D)$-directional stationary point.
\QEDA

\vgap

\noindent
	{\bf Proof of Proposition \ref{prop:subdir}}
Fix $x\in \dom \phi$. For simplicity, we denote
\begin{equation}\label{rel:phim}
    \phi_m(\cdot) = \phi_m(\cdot;x) \stackrel{\eqref{def:phim}}= \phi(\cdot) + \frac{m}{2}\|\cdot - x\|^2.
\end{equation}
 Since $\phi_m$ is convex, it follows from Theorem 3.26 of \cite{beck2017first} that
\begin{equation}\label{eq:directional}
    -\inf_{\|d\| \leq 1} \phi_m^{\prime}(x;d)
    = -\inf_{\|d\| \leq 1}\sigma_{\partial \phi_m(x)}(d)
    = - \inf_{\|d\| \leq 1} \sup_{s\in \partial \phi_m(x)}\inner{s}{d}
    = \inf_{s\in \partial \phi_m(x)}\|s\|
    =\operatorname{dist}(0; \partial \phi_m(x)).
\end{equation}
In view of \eqref{rel:phim}, \eqref{eq:key_chara1} becomes that $\partial \phi(x) = \partial \phi_m(x)$, which together with \eqref{eq:directional} implies that
\[
\operatorname{dist}(0; \partial \phi(x)) = \operatorname{dist}(0; \partial \phi_m(x)) = -\inf _{\|d\| \leq 1} \phi_m^{\prime}(x; d).
\]
Moreover, it follows from Definition \ref{def:di-deriv} and \eqref{rel:phim} that
\[
\phi_m'(x;d) = \liminf _{t \downarrow 0} \frac{\phi_m(x+td)-\phi_m(x)}{t} \stackrel{\eqref{rel:phim}}= \liminf _{t \downarrow 0} \frac{\phi(x+td) + \frac{m}{2}\|td\|^2-\phi(x)}{t} = \phi'(x;d).
\]
Therefore, the above two relations indicate that
\[
\operatorname{dist}(0; \partial \phi(x)) = -\inf _{\|d\| \leq 1} \phi^{\prime}(x;d),
\]
and the conclusion directly follows from Definition \ref{def:dsp}.
 \QEDA

\vgap
 
\noindent
	{\bf Proof of Proposition \ref{prop:ours}}
a)
Since $x$ is a
$(\bar \eta,\bar\varepsilon;m)$-regularized stationary point of $\phi$, 
there exists a pair $(w,\varepsilon) \in \R^n \times \R_{++}$ satisfying
\eqref{eq:app_sol}.
Using the fact that
$\phi_m(x;x) = \phi_{2m}(x;x)$ and
$\phi_m(\cdot;x) \le \phi_{2m}(\cdot;x)$, we easily see that
the inclusion in \eqref{eq:app_sol}
implies that
$w \in \partial_\varepsilon [\phi_{2m}(\cdot;x)](x)$.
Since  $\phi_{ m}$ is
convex in view of assumption (A1),
we easily see that
$\phi_{2m}(\cdot;x)$ is
$m$-strongly convex, and hence
the function $\phi_{2m}(\cdot;x)-\inner{w}{\cdot}$
has a global minimizer
$\tilde x$
which, in view of Lemma \ref{lem:mu convex} with $g=\phi_{2m}(\cdot;x)-\inner{w}{\cdot}$ and
$\mu=m $, satisfies
\begin{equation}
w \in \partial \left[ \phi_{2m}(\cdot;x) \right] (\tilde x), \quad \|x-\tilde x \| \le \sqrt{ \frac{2\varepsilon}{m} } \le \sqrt{ \frac{2 \bar \varepsilon}{m} }
\label{eq:nearby}
\end{equation}
where the last inequality is due to the
last inequality in \eqref{eq:app_sol}.
Now, letting
$\hat w := w - 2m (\tilde x-x)$, it follows from \eqref{eq:key_chara1} with $m = 2m$ and 
 Lemma \ref{lem:chara_weakly} with $(\varepsilon,x,c) = (0,\tilde x,x)$ that
\begin{equation}\label{eq:belong}
 \hat w \in \partial \left[ \phi_{2m}(\cdot;\tilde x) \right] (\tilde x)
= \partial \phi(\tilde x). 
\end{equation}
Moreover,
using the definition
of $\hat w$ and the triangle inequality, we
have
\[
\|\hat w\| \le \|w\| + 2m \|x-\tilde x\| \le
\bar \eta + 2 \sqrt{2m \bar \varepsilon},
\]
where the second inequality is due to the first inequality in
\eqref{eq:app_sol} and the inequality in \eqref{eq:nearby}. Then the above inequality and \eqref{eq:belong} imply that 
\begin{equation}\label{eq:tildedis}
\operatorname{dist}(0,\partial \phi(\tilde x)) \le \bar \eta + 2 \sqrt{2m \bar \varepsilon}.
\end{equation}
Now statement (a) follows from \eqref{eq:nearby}, \eqref{eq:tildedis}, and Proposition \ref{prop:subdir}.

b) Statement (b) immediately follows from the first statement and Proposition \ref{prop:relation}(a) with $\lam=1/m$.
\QEDA

% \begin{proof}
%     Denote $\hat x = \operatorname{prox}_{(1/2m)\phi}(x)$. Then from \eqref{eq:grad_M} with $\mu = 1/2m$ and the assumption that $x$ is a $(\varepsilon,1/(2m))$-Moreau stationary point, we know that
%    \begin{equation}\label{eq:norm_small}
%    2m\|x-\hat x\| \le \varepsilon.
%     \end{equation}
%    Since  $\phi_{ 2m}$ is
% convex in view of assumption (A1), using the optimality condition for minimizing $\phi_{2m}(\cdot;x)$ we have
% \[
% 0 \in   \partial \left[\phi_{2m}(\cdot;x) \right]  (\hat x ).
% \]
% Thus we know that
% \[
% 0 \in \partial_{\bar \varepsilon} [ \phi_{2m}(\cdot;x)](x)
% \]
% where 
% \[
%  \bar \varepsilon
% := \phi_{2m}(x;x) - \phi_{2m}(\hat x;x) =
% \phi(x) - \phi(\hat x) - m \|\hat x-x\|^2.
% \]
% Moreover, using the assumption that $\phi$ is $M$-Lipschitz, we can get
% \[
% \bar \varepsilon \le \phi(x) - \phi(\hat x)
% \le M \|\hat x-x\|.
% \]
% The statement now follows from the above inequality, \eqref{eq:norm_small} and the definition of $(0,M\varepsilon/(2m);2m)$-stationary point in Definition \ref{def:appr}.
% \end{proof}

\section{Results about an Important Recursive Formula}
\label{IRRF}

\begin{lemma}\label{lem:keyrecur}
		Assume that for some integer $J >0$, scalars $\delta>0$, $a \ge 0$ and $b \ge 0$ such that $a+b>0$, and
  scalars $\{q_0,\ldots,q_{J} \}$, we have that
		\begin{equation}\label{ineq:recursive1}
		q_{j} > \delta
  % =  q_{l} + \frac{q_{l}^2}{a
  % + b q_{l}} 
 \ \ \forall j \in \{0,\ldots,J\},
		\end{equation}
        and
        \begin{equation}\label{ineq:recursive2}
		q_{j}\left(1 + \frac{q_{j}}{a + b q_{j}}\right) 
  % =  q_{l} + \frac{q_{l}^2}{a
  % + b q_{l}} 
  \le q_{j-1}, \ \ \forall j \in \{1,\ldots,J\}.
		\end{equation}
		% \begin{itemize}\label{lem:waiting}
		% 	\item[a)] for any $l > 0$, we have
		% 	\begin{equation}
		% 	\label{ineq:induction}	
		% 	q_l   \le 
		% 	%\frac{\max\{2c, t_l\}}{j-l}=
		% 	\frac{2a +   2b t_{0} }{l};
		% 	\end{equation}
		% 	\item[b)] 
		% 	an index $l$ such that
		% 	$q_l \le 2a/b$ can be
		% 	found in a number of iterations bounded by
		% 	\[
		% 	1+\lceil 4b \rceil \log\left( \frac{b t_0}{a}\right).
		% 	\]
		% \end{itemize}
 Then, 
 % for any given $\delta>0$,
 % there exists an  index $\hat l \ge 1$ such that $q_{ \hat l} \le \delta$ and
  \begin{equation}\label{def:tl}
     J \le \bar N :=(2b+3) \log ^+ \left(\frac{  q_0}{\delta + a/(b+1)}\right) + \frac{2a}{\delta} + b +2.
  \end{equation}
  \end{lemma}
	\begin{proof}
 Define $\tilde \delta := \delta + a/(b+1)$ and
 $\theta := 1/[2(b+1)]$. Suppose for contradiction that $J > \bar N$.
 % Let
 % $\bar l := \sup \{ l \ge 1: q_l  \ge \tilde \delta\}$.
 We first claim that
\begin{equation}\label{ineq:j}
\bar j :=  \max \{ j \le J: q_j  \ge \tilde \delta\} \le {\cal T}:= \frac{\log ^ + (q_0/\tilde \delta)}{ \log (1+\theta) }\le (1+\theta^{-1})\log ^+\left(\frac{q_0}{\tilde \delta}\right)\le \bar N.
\end{equation}
% \[
% \frac{\theta}{1+\theta} \le \log (1+\theta) \le \theta
% \]
% ????? The claim is obvious if $q_0 < \tilde \delta$. To show the claim holds for $q_0 \ge \tilde \delta$ , 
Assume for contradiction that there exists an index $j > {\cal T} $ such
that $q_j \ge  \tilde \delta$.
Using the fact
$\{q_j\}$ is non-increasing (see \eqref{ineq:recursive2}) and the definitions of $\tilde \delta$ and $\theta$,
we have that for every  $l=1,\ldots,j$,
 \begin{align*}
 \frac{q_l}{a + b q_l} \ge \frac{\tilde \delta}{a + b\tilde \delta} = 
 \frac{\delta + a/(b+1)}{a + b\delta + ab/(b+1)} 
 \ge 
 \frac{\delta + a/(b+1)}{2 a + b\delta} 
 % = \frac{\delta (b+1)+a}{(b+1)(2a +b) \delta}
 \ge \frac{1}{2(b+1)} = \theta.
 \end{align*}
 Hence, 
 in view of  \eqref{ineq:recursive2}, we conclude that
 $
 (1+\theta)^j q_{j} \le
 (1+\theta)^{j-1} q_{j-1}
  \le  \cdots \le  q_0 
 $.
 On the other hand, using the
 fact that  $j > \cal T$ and the definition of $\cal T$ in \eqref{ineq:j}, we easily see that
 $q_0 < (1+ \theta)^j \tilde \delta$. Combining the last two conclusions, it follows that $q_j < \tilde \delta$.
 % where the strict inequality is due to the fact that  $j > \cal T$ and the definition of $\cal T$ in \eqref{ineq:j}. 
 Since the last conclusion contradicts the fact that $q_j \ge \tilde \delta$, the claim follows.
%  Thus we have
%  \begin{equation}\label{ineq:j}
%  j \le \log \frac{q_0}{\tilde \delta} / \log \frac{2b+3}{2b+2} \le 2(b+1) \log \frac{q_0}{\tilde \delta}.
% \end{equation}

It follows from \eqref{ineq:recursive2} that 
\[
q_{j-1} \ge q_j \frac{a+(b+1) q_j}{a+b q_j}.
\]
Since $a \ge 0$ and $ b \ge 0$, taking reciprocals gives
\[
\frac{1}{q_{j-1}} \le \frac{a+b q_j}{q_j\left(a+(b+1) q_j\right)}.
\]
Hence, we obtain
\begin{equation}\label{ineq:new}
    \frac{1}{q_j}-\frac{1}{q_{j-1}} \ge \frac{1}{q_j}-\frac{a+b q_j}{q_j\left(a+(b+1) q_j\right)}
=\frac{1}{a+(b+1) q_j}.
\end{equation}
Note that the definition
of $\bar j$ in \eqref{ineq:j} implies that
$\tilde \delta > q_{\bar j+1} \ge q_j$ for any $j \ge \bar j+1$. This fact and relation
 \eqref{ineq:new} then imply that for every $j \ge \bar j+1$,
   \[
   \frac{1}{q_{j}} - \frac{1}{q_{j-1}}   \ge \frac{1}{a
  + (b+1) \tilde \delta}.
   \]
   
   % where the second inequality is due to and the fact that $q_{l} \le q_{\bar l+1} < \tilde \delta$ for any $l \ge \bar l + 1$.
   Thus, for any given $j \ge \bar j+2$, summing the above inequality from $\bar j +2$ to $j$ and using the fact that $q_{\bar j+1} < \tilde \delta$,
   we then conclude that
   \[
   \frac{1}{q_j}  \ge \frac{1}{q_{\bar j +1}} +  \frac{j-1-\bar j }{a
  + (b+1) \tilde \delta} > \frac{1}{\tilde \delta} +  \frac{j-1-\bar j }{a
  + (b+1) \tilde \delta}  \ge  \frac{j-\bar j }{a
  + (b+1) \tilde \delta},
   \]
   and hence that
\begin{equation}\label{ineq:side}
        q_{\bar N} \le \frac{a
  + (b+1) \tilde \delta}{\bar N-\bar j } = \frac{2a + (b+1)\delta }{\bar N-\bar j } \le \delta
   \end{equation}
  where the equality is due to the definition of $\tilde \delta$ and the last inequality is due to the definition of $\bar N$ and $\bar j$ in \eqref{def:tl} and \eqref{ineq:j}, respectively. Hence, the above inequality contradicts  \eqref{ineq:recursive1} and thus the statement follows.
  % The conclusion of the lemma can now be easily seen to follow from \eqref{ineq:j} and 
  % \eqref{ineq:side}.
  % Thus using the definition of $\tilde l$ in \eqref{def:tl} and the first claim in \eqref{ineq:j} we have
  % \[q_{\tilde l}\le \frac{2a+(b+1)  \delta }{\tilde l-\bar l } \le \frac{2a + (b+1)\delta }{(2a + (b+1) \delta)/\delta} = \delta.
  % \]
  % which implies the conslusion of the lemma.
	\end{proof}

\section{Technical Results for the Proof of Theorem \ref{thm:main1}}
\label{APP:analysis}
The main result of this section is Lemma \ref{lem:bdd_smdist} which was used in the proof of Lemma \ref{lem:t1}. 
	
	Before stating and proving Lemma \ref{lem:bdd_smdist},
	we first present two technical results whose proof can be found in Appendix A of \cite{liang2024unified}.
	\begin{lemma}\label{lem:prox-dist}
	Let $z_0\in \R^n$, $0<\zeta < 1$, $\lam >0$, and $\Gamma\in \bConv{n}$ be given, and
	define
	\begin{equation}\label{eq:z_lam}
	    z_\lam=\underset{u\in \R^n}\argmin \left\lbrace \Gamma(u) +\frac{1}{2 \lam}\|u- z_0 \|^2 \right\rbrace,\quad
	    z_{\zeta \lam}=\underset{u\in \R^n}\argmin \left\lbrace \Gamma(u) +\frac{1}{2 \zeta \lam}\|u-z_0\|^2 \right\rbrace.
	\end{equation}
	Then, we have $ \|z_{\zeta\lam}-z_0\| \ge \zeta \|z_\lam-z_0\|$.
	\end{lemma}

	\begin{lemma}\label{lem:descent}
% 	{\bf Talk about $h$}
	For some $(M,L) \in \R^2_+$, assume that $(z_0,\lam)\in \R^n \times (0,1/ L)$,
	function
	 $ \tilde f  \in 
	 \bConv{n} \cap {\cal C}(M,L)$ and
	 $\Gamma, h \in \bConv{n} $ are such that
	 \[
	 \ell_{\tilde f}(\cdot;z_0)+h \le \Gamma \le \tilde f+h.
	 \]
	Then, for every $u\in \dom h$, we have
	\begin{equation}\label{ineq:descent}
	    \frac{1}{2\lam} \|u-z_\lam\|^2 + (\tilde f+h)(z_\lam) - (\tilde f+h)(u) \le \frac{1}{2\lam} \|u-z_0\|^2 + \frac{2\lam M^2}{1-\lam L}.
	\end{equation}
	\end{lemma}

 \begin{lemma}\label{lem:bdd_smdist}
	For some $(M,L) \in \R^2_+$ and $m \ge 0$, assume that $(z_0,\lam)\in \R^n \times \R_{++}$,
	function
	 $f  \in {\cal C}(M,L)$ is $m$-weakly convex, and
	 $\Gamma, h \in \bConv{n} $ are such that
	 \[
	 \ell_{ f}(\cdot;z_0)+h \le \Gamma \le f_m(\cdot;z_0)+h,
	 \]
	where $f_m(\cdot;z_0)$ is as in \eqref{def:phim}.
	Then,  for every $u \in \R^n$, we have
	\begin{equation}
	\zeta^2\left(\frac{1}{4\zeta \lam}+\frac{m}{2} \right)\|z_\lam - z_0 \|^2 
	\le
	(f+h)(u) - (f+h)(z_{\zeta\lam}) + \left(\frac m2+\frac{1}{\zeta \lam} \right) \|u-z_0\|^2 + 4\zeta \lam M^2,
\label{eq:uniform_bdd}
	\end{equation}
	where $z_\lam$ and $z_{\zeta \lam}$ are as in \eqref{eq:z_lam} and $\zeta $ is defined as \eqref{label:def_zeta}.
\end{lemma}

	\begin{proof}
	We first prove \eqref{eq:uniform_bdd} under the assumption that
	$\lam \in (0,(2(L+m))^{-1}]$,
	and hence $\zeta=1$ in view of \eqref{label:def_zeta}.
	Indeed, noting that
	$f_m(\cdot;z_0) \in \Conv{n} \cap {\cal C}(M,L+m)$ due to Lemma \ref{lem:Lip_ineq}, it follows from Lemma~\ref{lem:descent} with
	$\tilde f = f_m(\cdot;z_0)$  and $ \lam \le 1/[2(L+m)]$ that
	for any $u \in \dom h$:
    \[
	    \frac{1}{2 \lam} \|u- z_{\lam}\|^2 + (f_m(\cdot;z_0)+h)(z_{\lam}) - (f_m(\cdot;z_0)+h)(u) - \frac{1}{2 \lam}  \|u-z_0\|^2 
	    \stackrel{\eqref{ineq:descent}}\le \frac{2  \lam  M^2}{1- \lam ( L +m)}
	    \le  4\lam M^2,
	\]
	and hence that
	% \begin{align*}
	% (f+h)(u) - (f+h)(z_{\lam}) &+ \frac{m}{2}\|u-z_0\|^2+4 \lam M^2
	% \ge \frac{1}{2\lam} \left(\|u- z_{\lam}\|^2 - \|u-z_0\|^2 \right) + \frac{m}{2}\| z_{\lam} - z_0\|^2 \\
	% &=  \frac{1}{2}\left(\frac{1}{ \lam} + m\right)\|z_{\lam}- z_0\|^2 -\frac{1}{\lam}\inner{z_{\lam} - z_0}{u - z_0} \\
	% &\ge \frac{1}{2}\left(\frac{1}{ \lam} + m\right)\|z_{\lam} - z_0\|^2 -\frac{1}{2  \lam} \left(\frac{1}{2}\|z_{\lam} - z_0\|^2 + 2\|u - z_0\|^2 \right) 
	% \end{align*}
    \[
	(f_m(\cdot;z_0)+h)(u) - (f_m(\cdot;z_0)+h)(z_{\lam}) +4 \lam M^2
	\ge \frac{1}{2\lam} \left(\|u- z_{\lam}\|^2 - \|u-z_0\|^2 \right) 
	\ge \frac{1}{4 \lam}\|z_{\lam} - z_0\|^2 -\frac{1}{\lam} \|u - z_0\|^2,
	\]
	where the last inequality is due to the fact that
	$2a^2+2b^2 \ge (a+b)^2$ and the triangle inequality.
	Rearranging the above inequality and using the definition of $f_m(\cdot;z_0)$ in \eqref{def:phim}, we then conclude that
	\begin{equation}
	   (f+h)(u) - (f+h)(z_{\lam}) + 4\lam M^2\ge  \left(\frac{1}{ 4\lam} + \frac{m}{2}\right)\|z_{\lam} - z_0\|^2 - \left( \frac1{\lam} +\frac{m}2 \right)  \|u - z_0\|^2
	   \label{eq:small_lam}
	\end{equation}
	which, in view of the fact that
	$\zeta=1$, immediately implies \eqref{eq:uniform_bdd}.
	Next, we show that \eqref{eq:uniform_bdd} also holds for $\lam >1/[2(L+m)]$.
Noting that \eqref{label:def_zeta} implies that $\zeta \in (0,1)$ and $\zeta \lam = 1/[2(L+m)]$, it then follows from \eqref{eq:small_lam} with
$\lam$ replaced by $\zeta \lam$ that
	\[
	    (f+h)(u) - (f+h)(z_{\zeta\lam}) + \left( \frac{m}2 +\frac{1}{\zeta \lam} \right)\|u -z_0\|^2+ 4 \zeta\lam M^2 \ge
	    \left(\frac{1}{4 \zeta \lam}+\frac{m}{2} \right)\|z_{\zeta \lam} - z_0\|^2.
	\]
    Finally, \eqref{eq:uniform_bdd} immediately follows from the above inequality and Lemma \ref{lem:prox-dist}.
	\end{proof}

 \section{Further Discussion Regarding \eqref{ineq:Moreau-bound}} \label{sec:termination}

\begin{proposition}
   For every $K \ge 1$, then we have
    \[
 \inf_{k \le K} \|\nabla \hat M^{\lam}(\hat y_{k-1})\|^2 
    \le \frac{2(1+\lam m)^2}{\lam} \left(4 \delta + \frac 3K [\phi(\hat x_0) - \phi(\hat y_K)]\right).
\]
\end{proposition}

\begin{proof}
 For every $k\ge 1$, let $x_k^*$ be the minimizer of \eqref{sub_pro}. Hence, using the definition of $\phi_m(\cdot;\hat y_{k-1})$ in \eqref{def:phim} and the fact that the objective function of \eqref{sub_pro} and is $\lam^{-1}$-strongly convex, we have
\begin{equation}\label{ineq:phim}
    \phi_{m}(\hat y_k;\hat y_{k-1})+ \frac{1}{2\lam} \|\hat y_{k} - \hat y_{k-1}\|^2  \ge \phi_{ m}(x_k^*;\hat y_{k-1}) + \frac{1}{2\lam} \|x_{k}^* - \hat y_{k-1}\|^2 + \frac{1}{2\lam} \|\hat y_k-x_k^*\|^2.
\end{equation}
Moreover, using Lemma \ref{lem:basic_prop}(a) and the fact that the objective function in \eqref{eq:optmality} is $\lam^{-1}$-strongly convex, we for every $u\in \dom h$,
\begin{equation}\label{ineq:hatGamma}
    \hat \Gamma_k(\hat x_k) + \frac1{2\lam} \|\hat x_k-\hat y_{k-1}\|^2 \le
\hat \Gamma_k(u) + \frac1{2\lam} \|u-\hat y_{k-1}\|^2 - \frac1{2\lam} \|u-\hat x_k\|^2.
\end{equation}
Combining \eqref{eq:hattheta} and \eqref{eq:control}, and using \eqref{ineq:hatGamma}, we have for every $u\in \dom h$,
\begin{align}
    &\hat \delta_k \overset{\eqref{eq:hattheta},  \eqref{eq:control}}\ge \phi_{ m}(\hat y_k;\hat y_{k-1})+ \frac{1}{2\lam} \|\hat y_{k} - \hat y_{k-1}\|^2  - \hat \Gamma_k(\hat x_k) - \frac{1}{2\lambda} \| \hat x_k - \hat y_{k-1}  \|^2 \nn \\
    &\stackrel{\eqref{ineq:hatGamma}}\ge 
 \phi_{ m}(\hat y_k;\hat y_{k-1})+ \frac{1}{2\lam} \|\hat y_{k} - \hat y_{k-1}\|^2 -\hat \Gamma_k(u) - \frac1{2\lam} \|u-\hat y_{k-1}\|^2 + \frac1{2\lam} \|u-\hat x_k\|^2 \nn \\
 &\stackrel{\eqref{eq:basic}}\ge 
 \phi_{ m}(\hat y_k;\hat y_{k-1})+ \frac{1}{2\lam} \|\hat y_{k} - \hat y_{k-1}\|^2 - \phi_m (u;\hat y_{k-1}) - \frac1{2\lam} \|u-\hat y_{k-1}\|^2 + \frac1{2\lam} \|u-\hat x_k\|^2, \label{ineq:second-step}
\end{align}
where the last inequality is due to the inequality in \eqref{eq:basic}.
Taking $u= \hat y_{k-1}$ in \eqref{ineq:second-step} yields
\begin{equation}\label{ineq:more-step}
    \hat \delta_k + \phi(\hat y_{k-1}) - \phi(\hat y_k) \ge \hat \delta_k + \phi(\hat y_{k-1}) - \phi_m(\hat y_k;\hat y_{k-1}) \ge \frac{1}{2\lam}\|\hat y_k-\hat y_{k-1}\|^2 + \frac{1}{2\lam}\|\hat x_k-\hat y_{k-1}\|^2.
\end{equation}
Taking $u= x_k^*$ in \eqref{ineq:second-step} and using \eqref{ineq:phim}, we have
\begin{equation}\label{ineq:third-step}
    \hat \delta_k \ge 
  \frac1{2\lam} \|\hat y_k-x_k^*\|^2 + \frac1{2\lam} \|x_k^*-\hat x_k\|^2.
\end{equation}
Combining \eqref{ineq:more-step} and \eqref{ineq:third-step},
we obtain
\begin{align}
    2\hat \delta_k + \phi(\hat y_{k-1}) - \phi(\hat y_k) 
    & \ge \left(\frac{1}{2\lam}\|\hat x_k-\hat y_{k-1}\|^2
+ \frac1{2\lam} \|x_k^*-\hat x_k\|^2\right) + \left(\frac{1}{2\lam}\|\hat y_k-\hat y_{k-1}\|^2 + \frac{1}{2\lam}\|\hat y_k-x_k^*\|^2\right) \nn \\
&\ge \frac1{4\lam} \|x_k^*-\hat y_{k-1}\|^2 + \frac{1}{4\lam} \|x_k^*-\hat y_{k-1}\|^2 = \frac{1}{2\lam}\|x_k^*-\hat y_{k-1}\|^2, \label{ineq:xy}
\end{align}
where the second inequality is due to the fact that $2a^2+2b^2\ge(a+b)^2$ for every $a, b \in \R$ and the triangle inequality.
Noting from step 3) of PBF and \eqref{def:wj}, and using \eqref{eq:est_w}, we have
\[
\hat \delta_k \stackrel{\eqref{def:wj}}= \delta +\frac{\lam}{8(m \lam +1)}\| \hat w_k\|^2 \stackrel{\eqref{eq:est_w}}\le 2\delta +  \frac12 \left( m+\frac{1}{\lambda} \right) \|\hat y_{k} - \hat y_{k-1}\|^2 \stackrel{\eqref{eq:bdd_diff}}\le 2\delta + \phi(\hat y_{k-1}) - \phi(\hat y_k),
\]
where the last inequality is due to \eqref{eq:bdd_diff}.
The above inequality and \eqref{ineq:xy} yield that
\[
\frac{1}{2\lam}\|x_k^*-\hat y_{k-1}\|^2 \le 4\delta + 3[ \phi(\hat y_{k-1}) - \phi(\hat y_k)].
\]
Moreover, recall that $x_k^*$ is the minimizer of \eqref{sub_pro}, and hence $x_k^* = \hat x^{\lam}(\hat y_{k-1})$ in view of \eqref{eq:hatx}. It thus follows from \eqref{eq:new_grad} that
\[
\frac{\lam}{2(1+\lam m)^2} \|\nabla \hat M^{\lam}(\hat y_{k-1})\|^2 \le 4\delta + 3[ \phi(\hat y_{k-1}) - \phi(\hat y_k)].
\]
Therefore, the proposition immediately follows by summing the above inequality from $k=1$ to $K$ and dividing the resulting inequality by $K$.
\end{proof}

\end{document}